\documentclass[12pt, reqno]{amsart}
\usepackage[T1]{fontenc}
\usepackage{amsmath,amssymb,color,xcolor,graphicx,mathtools}
\usepackage{fullpage} 
\usepackage[backref=page,breaklinks]{hyperref}
\usepackage{colonequals}
\usepackage{mathrsfs}
\usepackage{nicefrac}
\usepackage[inline, shortlabels]{enumitem}
\hypersetup{colorlinks=true,urlcolor=blue,citecolor=blue,linkcolor=blue}
\usepackage{comment}
\usepackage{subcaption}
\usepackage{tikz-cd}
\usepackage{anyfontsize}
\usepackage{lscape}
\usepackage[colorinlistoftodos]{todonotes}
\usepackage{xstring}
\usepackage[nameinlink]{cleveref}
\Crefname{section}{\S\!\!}{\S\!\!}

\numberwithin{equation}{subsection}

\theoremstyle{plain}
\newtheorem{lemma}[equation]{Lemma}
\newtheorem{prop}[equation]{Proposition}

\newtheorem{corollary}[equation]{Corollary}
\newtheorem{theorem}[equation]{Theorem}
\newtheorem{thm}[equation]{Theorem}

\newtheorem{conj}[equation]{Conjecture}

\theoremstyle{remark}
\newtheorem{remark}[equation]{Remark}

\newtheorem{example}[equation]{Example}

\theoremstyle{definition}

\newenvironment{enumalph}
{\begin{enumerate}}
{\end{enumerate}}

\newenvironment{enumroman}
{\begin{enumerate}}
{\end{enumerate}}

\newcommand{\A}{\mathbb A}
\let\C\relax
\newcommand{\C}{\mathbb{C}}
\newcommand{\F}{\mathbb F}

\newcommand{\PP}{\mathbb P}
\newcommand{\Qbar}{\mathbb{Q}^{\textup{al}}}
\newcommand{\Q}{\mathbb Q}
\newcommand{\R}{\mathbb R}

\newcommand{\Z}{\mathbb Z}

\newcommand{\fraka}{\mathfrak a}
\newcommand{\frakb}{\mathfrak b}
\newcommand{\frakc}{\mathfrak c}

\newcommand{\frakD}{\mathfrak D}

\newcommand{\frakh}{\mathfrak h}
\newcommand{\frakl}{\mathfrak l}
\newcommand{\frakn}{\mathfrak n}

\newcommand{\frakN}{\mathfrak N}
\newcommand{\frakp}{\mathfrak p}

\newcommand{\frakq}{\mathfrak q}

\newcommand{\calH}{\mathcal H}
\newcommand{\calO}{\mathcal O}

\newcommand{\h}[2]{\frakh^{#1}_{#2}}
\newcommand{\hSiegel}[1]{\calH_{#1}}

\newcommand{\thetacha}[3]{\vartheta_{#1}\hspace{-0.5mm}\big[\hspace{-0.5mm}\begin{smallmatrix}
#2\\#3
\end{smallmatrix}\hspace{-0.5mm} \big]}

\newcommand{\Magma}{\textsf{Magma}}

\newcommand{\injects}{\hookrightarrow}

\newcommand{\tensor}{\otimes} %

\newcommand{\Fhat}{\widehat{F}}

\newcommand{\Gammahat}{\widehat{\Gamma}}

\newcommand{\Zhat}{\widehat{\mathbb Z}}

\newcommand{\Lchi}[1]{L(\tau_j(f), #1, \chi)}

\newcommand{\prodprime}{\sideset{}{^{'}}{\prod}}

\DeclareMathOperator{\Aut}{Aut}
\DeclareMathOperator{\Cl}{Cl}

\DeclareMathOperator{\Char}{char}

\let\det\relax %

\DeclareMathOperator{\det}{det}
\DeclareMathOperator{\disc}{disc}
\DeclareMathOperator{\End}{End}
\DeclareMathOperator{\Frob}{Frob}
\DeclareMathOperator{\Galois}{Gal}
\newcommand{\Gal}[2]{\Galois({#1}\,|\,{#2})}
\DeclareMathOperator{\GL}{GL}
\DeclareMathOperator{\img}{img}

\DeclareMathOperator{\Jac}{Jac}
\DeclareMathOperator{\M}{M}
\DeclareMathOperator{\nrd}{nrd}

\DeclareMathOperator{\rmP}{P}
\DeclareMathOperator{\PGL}{PGL}

\DeclareMathOperator{\PSL}{PSL}
\DeclareMathOperator{\sgn}{sgn}
\DeclareMathOperator{\SL}{SL}

\DeclareMathOperator{\Sym}{Sym}

\DeclareMathOperator{\tr}{tr}
\DeclareMathOperator{\Tr}{Tr}
\DeclareMathOperator{\Nm}{Nm}

\newcommand{\GammaL}{\mathop{\Gamma{\rm L}}\nolimits}

\newcommand{\PGammaL}{\mathop{{\rm P}\Gamma{\rm L}}\nolimits}

\newcommand{\defi}{\textsf}

\newcommand{\lmfdbnf}[1]{%
    \href{https://www.lmfdb.org/NumberField/#1}{\textsf{#1}}}
\newcommand{\lmfdbform}[1]{%
    \StrBefore{#1}{-}[\fieldlabel]%
    \href{https://www.lmfdb.org/ModularForm/GL2/TotallyReal/\fieldlabel/holomorphic/#1}{\textsf{#1}}}
\newcommand{\lmfdbshortform}[2]{%
    \href{https://www.lmfdb.org/ModularForm/GL2/TotallyReal/#1/holomorphic/#1-#2}{\textsf{#2}}}
\definecolor{caribbeangreen}{rgb}{0.0, 0.8, 0.6}

\begin{document}

\setcounter{tocdepth}{1}

\title[17T7 is a Galois group]{The constructive inverse Galois problem via Hilbert modular forms: realizing the transitive group 17T7}
\date{\today}

\author[vanBommel]{Raymond van Bommel}
\address{
    Department of Mathematics,
      Massachusetts Institute of Technology,
      77 Massachusetts Ave.,
      Bldg. 2-336
      Cambridge,
      MA 02139,
      USA;     School of Mathematics,
    University of Bristol,
    Fry Building, Woodland Road, Bristol, BS8 1UG, United Kingdom 
}
\email{\href{mailto:r.vanbommel@bristol.ac.uk}{r.vanbommel@bristol.ac.uk}}
\urladdr{\url{https://raymondvanbommel.nl/}}

\author[Costa]{Edgar Costa}
\address{
  Department of Mathematics,
  Massachusetts Institute of Technology,
  77 Massachusetts Ave.,
  Bldg. 2-336
  Cambridge,
  MA 02139,
  USA
}
\email{\href{mailto:edgarc@mit.edu}{edgarc@mit.edu}}
\urladdr{\url{https://edgarcosta.org}}

\author[Elkies]{Noam D. Elkies}
\address{
 Department of Mathematics, 
 Harvard University, 
 Cambridge, MA
 02138,
 USA
}
\email{\url{elkies@math.harvard.edu}}
\urladdr{\url{https://people.math.harvard.edu/~elkies/}}

\author[Keller]{Timo Keller}
\address{
Institut für Mathematik, Universität Würzburg, Emil-Fischer-Strasse 30, 97074,
Würz\-burg, Germany; 
Rijksuniveriteit Groningen, Bernoulli Institute, Bernoulliborg, Nijenborgh 9, 9747 AG Groningen, The Netherlands; 
Leibniz Universität Hannover, Institut für Algebra, Zahlentheorie und Diskrete Mathematik, Welfengarten 1, 30167 Hannover, Germany}
\email{\href{mailto:math@kellertimo.de}{math@kellertimo.de}}
\urladdr{\url{https://www.timo-keller.de}}

\author[Schiavone]{Sam Schiavone}
\address{
  Massachusetts Institute of Technology,
  Department of Mathematics,
  77 Massachusetts Ave.,
  Bldg. 2-336
  Cambridge,
  MA 02139,
  USA
}
\email{\href{mailto:sam.schiavone@gmail.com}{sam.schiavone@gmail.com}}
\urladdr{\url{https://math.mit.edu/~sschiavo/}}

\author[Voight]{John Voight}
\address{
  Dartmouth College,
  Department of Mathematics,
  6188 Kemeny Hall, 
  Hanover, NH 03755-3551,
  USA;
  Department of Mathematics and Statistics,
  Carslaw Building (F07),
  University of Sydney, NSW 2006, Australia
}
\email{\href{mailto:jvoight@gmail.com}{jvoight@gmail.com}}
\urladdr{\url{https://jvoight.github.io}}


\begin{abstract}
We show how Hilbert modular forms can be used in the constructive inverse Galois problem over the rationals.  In particular, we prove that the transitive permutation group \textsf{17T7}, isomorphic to a split extension of $C_2$ by $\PSL_2(\F_{16})$, is a Galois group over the rationals and exhibit an explicit degree $17$ polynomial with this Galois group.  The group arises from the field of definition of the $2$-torsion on an abelian fourfold with real multiplication defined over a real quadratic field; we find such a fourfold attached to a Hilbert modular form.  Building upon work of Demb\'el\'e, we describe a method for reconstructing a period matrix attached to a Hilbert modular form, and we use it to construct the $2$-isogeny polynomial. We also rigorously identify the relevant fourfold as the Jacobian of a genus $4$ Shimura curve and compute explicit equations for this curve.
\end{abstract}

\maketitle

\tableofcontents

\section{Introduction}

\subsection{Motivation and first results}

A question of enduring fascination, the Inverse Galois Problem (IGP) \cite{Serre,MM,JLY} asks whether every finite group is realizable as a Galois group over~$\Q$.
Here we will be interested in the effective IGP, where given a transitive subgroup $G \leq S_d$ (up to conjugation), one asks further for an explicit polynomial $f(x) \in \Q[x]$ of degree $d$ whose Galois group, as a permutation group via the action on the roots, is equivalent to $G$.  

Except for two intransigent groups, the effective IGP has a positive answer for every transitive group $G \leq S_d$ with $d \leq 23$~\cite{Dokchitser,KlunersMalle,KMweb}.  Of these two exceptions, the most notorious is the sporadic simple group~$M_{23}$, the Mathieu group of order~$10\,200\,960$.  Although not realized over $\Q$, the group~$M_{23}$ has been realized as a Galois group over any number field~$K$ where $-1$ is a sum of two squares in $K$~\cite{Granboulan}.

The remaining exception, the one of smallest transitive degree, is the group $G$ with label \href{http://www.lmfdb.org/GaloisGroup/17T7}{\textsf{17T7}} and order $\#G = 8160 = 2^5 3^1 5^1 17^1$.  The group $G$ is isomorphic to $\PSL_2(\F_{16}) \rtimes C_2$ and so fits into a split exact sequence
\begin{equation} 
    1 \to \PSL_2(\F_{16}) \to G \to C_2 \to 1,
\end{equation}
where the nontrivial element $\sigma$ of $C_2$ 
acts entrywise by the element of $\Gal{\F_{16}}{\F_2}$ of order $2$ (i.e., by $a \mapsto a^4$).  
We obtain a permutation representation $G \hookrightarrow S_{17}$ via the natural action on $\PP^1(\F_{16})$, with $\sigma$ again acting entrywise.

We first record a general criterion that can be used to solve 
the inverse Galois problem for groups of the shape $\GL_n(k) \rtimes C_m$ and $\SL_n(k) \rtimes C_m$, where $k$ is a finite field of characteristic $\ell$.  For notation, we refer to \cref{sec:descent}.

\begin{thm}[\Cref{thm:KF0Gal}] \label[theorem]{thm:intromain}
Let $F_0 \subseteq F$ be a Galois extension of number fields.
Let $\rho \colon \Galois_F \to \GL_n(k)$ be a Galois representation, let $G \colonequals \rho(\Galois_{F})$ be the image and suppose $G=\GL_n(k)$ or $G=\SL_n(k)$.  Suppose that there exists an injective group homomorphism $\tau \colon \Gal{F}{F_0} \to \Aut(k)$ such that $\rho^\sigma \simeq \tau_\sigma(\rho)$ for all $\sigma \in \Gal{F}{F_0}$.  If $G=\SL_n(k)$, further suppose that $\gcd([F:F_0],n)$ is a power of $\ell$.

Then the extension $L=F(\rho) \supseteq F_0$ is Galois with $\Gal{L}{F_0} \cong G \rtimes \Gal{F}{F_0}$, where $\Gal{F}{F_0}$ acts on $G \leq \GL_n(k)$ entrywise via the map~$\tau$.
\end{thm}

Specializing to the case where $\rho$ is the semisimple mod $\frakl$ Galois representation attached to a Hilbert modular form $f$ and applying the Brauer--Nesbitt theorem, we obtain a criterion purely in terms of the Hecke eigenvalues of $f$ (\Cref{mainthm}).
In order to realize \textsf{17T7} as a Galois group, we find a suitable Hilbert modular form and apply the criterion to its associated mod~$2$ Galois representation.

We go beyond an existence proof and construct such an extension explicitly.

\begin{thm} \label[theorem]{thm:17t7main}
The group $G=\textup{\textsf{17T7}}$ is a Galois group over $\Q$.  More precisely, the polynomial
\begin{equation} \label{eqn:fx17}
\begin{aligned}
    x^{17} &- 2 x^{16} + 12 x^{15} - 28 x^{14} + 60 x^{13} - 160 x^{12} + 200 x^{11} - 500 x^{10} + 705 x^{9} - 886 x^{8}\\
    &+ 2024 x^{7} - 604 x^{6} + 2146 x^{5} + 80 x^{4} - 1376 x^{3} - 496 x^{2} - 1013 x - 490 \in \Q[x]
\end{aligned}
\end{equation}
has Galois group $G$.  
\end{thm}

The polynomial \eqref{eqn:fx17} was found by constructing the mod $2$ Galois representation attached to a Hilbert modular form; more precisely, we have the following relation.    

\begin{thm} \label[theorem]{thm:explicitg4}
Let $X_0$ be the smooth projective curve of genus $4$ over $\Q$ birational to the affine curve defined by the equation
\begin{equation} \label{eqn:singmod}
\begin{aligned}
&8 x^4 y + 8 x^3 y^2 + 10 x^2 y^3 + 4 x y^4 + 2 y^5 - 
    8 x^4 + 24 x^3 y + 6 x^2 y^2 + 12 x y^3 \\
    &\qquad - 
    2 y^4 - 2 x^3 - 3 x^2 y + 6 x y^2 - 
    11 y^3 + 14 x^2 - 6 x y - 4 y^2 - 
    7 x + 2 y + 1 = 0.
\end{aligned}
\end{equation}
Then the Jacobian of $X_0$ is isogenous over $F=\Q(\sqrt{3})$ to the abelian variety $A$ attached to the Hilbert modular form \lmfdbform{2.2.12.1-578.1-d}, and $\Q(A[2])$ is the splitting field of the polynomial \eqref{eqn:fx17}.
\end{thm}

In the rest of the paper, we will explain the methods, which rely on the theory of Hilbert modular forms and Oda's conjecture, used to find the polynomials in \eqref{eqn:fx17} and \eqref{eqn:singmod}.

\subsection{Overview} \label[section]{sec:overview} 

There has been substantial work using classical modular forms to solve the IGP for simple groups of the form $\PSL_2(\F_{q})$; see e.g., Zywina~\cite{Zywina} for a recent advance and many references.  In principle, these methods are also effective (computable in deterministic polynomial time), due to work of Edixhoven~\cite[Theorem~14.1.1]{Edixhoven} and others; calculations have been carried out by Bosman~\cite{Bosman}, Mascot~\cite{Mascot}, and again many others.  Although this approach to the IGP admits many variations, a recurring theme is to exhibit Galois groups over~$\Q$ via their action on torsion points of modular abelian varieties over~$\Q$, which appear as quotients of the Jacobian of a modular curve.  

A natural extension of this method works with abelian varieties and modular forms over number fields $F$.  When $F$ is a totally real field, we may work with Hilbert modular forms in a manner analogous to the classical case~\cite{DV}. This gives rise to Galois extensions of~$F$. However, in general the Galois groups over~$\Q$ obtained from the normal closure are wreath products rather than semidirect products.  

The aforementioned explicit criterion (\Cref{mainthm}) allows us to descend from $F$ to~$\Q$, yielding Galois groups that are (subgroups of) semidirect products.  This criterion takes advantage of additional symmetries observed by Gross in work of Demb\'el\'e~\cite{Dembele-peq2} and appearing in work of Demb\'el\'e--Greenberg--Voight \cite[\S\,1]{DGV}; see~\Cref{rmk:galalign}.  We then use this to solve the IGP for \textsf{17T7} by exhibiting certain Hilbert modular forms over abelian totally real fields $F$, see 
\Cref{thm:17T7ineff}.  Moreover, this method applies to many other groups $G$ that are split extensions of finite cyclic groups by $\PGL_2(\F_q)$; see \Cref{rmk:other groups}.

Our final task then is the effective resolution of the IGP for these groups.
In principle, when the Hilbert modular form arises via the Jacquet--Langlands correspondence on a Shimura curve, it should be possible to generalize the work above from the case of classical modular curves.
However, such an approach in practice poses theoretical limitations and substantial computational challenges (some aspects of which we hope to return to in future work).
Instead, what is needed is a version of the Eichler--Shimura construction for Hilbert modular forms suitable for computation, attaching to a Hilbert modular newform $f$ of weight $2$ an abelian variety $A_f$ with matching Galois representation and $L$-function.
In work of Demb\'el\'e~\cite{Dembele}, a numerical approach was outlined in the special case where $F$ is a real quadratic field of narrow class number one and $f$ has rational coefficients, so that $A_f$ is an elliptic curve.
This algorithm computes the period lattice assuming a conjecture of Oda~\cite{Oda} as refined by Darmon--Logan~\cite{DL}; see~\Cref{conj:oda}.  

One of our contributions is to generalize this approach, allowing arbitrary narrow class number and coefficient field.
Given the Hilbert modular newform $f$ over the Galois totally real field $F$, with coefficient field $K_f$ of degree $g=[K_f:\Q]$, the rough outline is as follows.
\begin{enumerate}[1.]
    \item 
      Compute periods for $A_f$ by computing $L(f, 1,\chi)$ for many quadratic characters $\chi$. 
    \item
      Construct the moduli point $z \in (\C\smallsetminus \R)^g$ corresponding to $A_f$ as ratios of the periods, and form the corresponding period matrix $\Pi$.
    \item
      Repeat for the conjugates of $f$ under $\Gal{F}{\Q}$.  
\end{enumerate}
We then form suitable polynomials in the theta constants with characteristic evaluated at $\Pi$ and its conjugates. Moreover, when the period matrix lies in the Schottky locus, we can also seek to reconstruct the abelian variety as the Jacobian of a curve.  
Several new features arise in this generalization, as is perhaps clear from this description.

We carry out this approach to solve the effective IGP for \textsf{17T7}.
Specifically, we choose a Hilbert modular form over $\Q(\sqrt{3})$ with LMFDB label \lmfdbform{2.2.12.1-578.1-d} of level norm $578$.  
In this case, there is a Shimura curve (\Cref{prop:descendtoQQ}) whose Jacobian is isogenous over $F$ to the abelian fourfold $A$ attached to this form, and we give the explicit realization in \Cref{thm:explicitg4}, combining \Cref{prop:modularity} and \Cref{prop:notnicolas}.

\subsection{Structure of the article}

In~\cref{sec:existence}, we dive into our descent approach to the IGP.
We then exhibit our general  method for the Eichler--Shimura construction in~\cref{sec:constructive}.
We conclude in~\cref{sec:17T7} by applying our methods to the particular modular form that produced the \textsf{17T7}-polynomial in~\Cref{thm:17t7main}.

\subsection{Acknowledgements}

The authors thank the organizers, Jennifer Balakrishnan, Bjorn Poonen, and Akshay Venkatesh, of the PCMI 2022 program on Number Theory Informed by Computation.  The idea to use Hilbert modular forms was also noted early by Pip Goodman, and an initial list of candidates was compiled in collaboration with him; we acknowledge and thank him for his contribution.
We also thank Maarten Derickx for his initial contribution to the project, specifically regarding his work determining the necessary properties of Hilbert modular forms used in our method.  
We thank Thomas Bouchet for his help in computing the curve \eqref{eqn:genus4curve}, Lassina Demb\'el\'e for suggesting the argument that proves modularity of this curve (\Cref{prop:modularity}), and Andreas-Stephan Elsenhans for help with a Galois group computation (\Cref{prop:notnicolas}).  
We also thank the inspiring lectures of Tim Dokchitser in the PCMI 2021 Graduate Summer School, Number Theory Informed by Computation, which brought this problem to the authors' attention.
We thank Nicolas Mascot for his computational verification in \Cref{rmk:nicolas} and for his comments on \Cref{prop:notnicolas}.
We also thank Sachi Hashimoto, Adam Logan, and many other participants of the PCMI 2022 program who contributed to the early stages of the project. 
Finally, we thank the anonymous referees for their valuable suggestions leading to several improvements of this article.

Van Bommel, Costa, and Schiavone were supported by a Simons Collaboration grant (550033).
Van Bommel has additionally been supported by C\'eline Maistret’s Royal Society Dorothy Hodgkin Fellowship.
Costa has additionally been supported by a Simons Foundation grant (SFI-MPS-Infra\-structure-00008651, AS).
Elkies was supported by a Simons Collaboration grant (550031).
Keller was supported by the 2021 MSCA Postdoctoral Fellowship 01064790 -- Ex\-pli\-cit\-Rat\-Points. 
Voight was supported by grants from the Simons Foundation: (550029, JV) and (SFI-MPS-Infra\-structure-00008650, JV).

\section{Inverse Galois problem via Hilbert modular forms} \label[section]{sec:existence}

In this section, we describe an approach to the Inverse Galois Problem for groups that are split extensions of finite cyclic groups by $\GL_2(\F_q)$, as well as certain subgroups and quotients, using Hilbert modular forms.  The main result is the criterion in~\Cref{mainthm}.  

\subsection{Group theory setup}

Let $k$ be a finite field of characteristic $\Char k = \ell$ with prime field $k_0 \subseteq k$.  Let $A \leq k^\times$ be a subgroup.  We define
\begin{equation} \label{eqn:GL2kA}
\GL_2(k)_A \colonequals \{g \in \GL_2(k) : \det g \in A\}
\end{equation}
for the subgroup of matrices whose determinant lies in $A$.  We write 
\begin{equation}
\rmP \colon \GL_2(k) \to \PGL_2(k)
\end{equation}
for the canonical projection, and for $G \leq \GL_2(k)$, we write $\rmP\! G \leq \PGL_2(k)$ for the projective image.  The map $u\colon \GL_2(k) \to k$ defined by
\begin{equation}
u(g) \colonequals (\tr g)^2/(\det g) 
\end{equation}
satisfies $u(cg) = u(g)$ for all $c\in k^*$ and $g \in \GL_2(k)$,
and thus descends to a map $\PGL_2(k) \to k$ that we also denote by~$u$.
This map is constant on conjugacy classes, and is surjective because
$u\Big(\begin{pmatrix} 0 & 1 \\ 1 & 0 \end{pmatrix}\Big) = 0$ and 
$u\Big(\begin{pmatrix} v & -v \\ 1 & 0 \end{pmatrix}\Big) = v$ for any $v \in k^*$.

As usual, we write $\SL_2(k)\trianglelefteq \GL_2(k)$ (taking $A=\{1\}$) for the determinant $1$ subgroup and $\PSL_2(k)=\SL_2(k)/\{\pm 1\}$.  When $\ell$ is odd, we have $\PSL_2(k) \leq \PGL_2(k)$ of index $2$;  otherwise (when $\ell=2$) we have $\SL_2(k)=\PSL_2(k)=\PGL_2(k)$.  
Finally, we have $\PGL_2(k)_A = \PSL_2(k)$ if $A \leq (k^{\times})^2$; otherwise $\PGL_2(k)_A=\PGL_2(k)$.

\begin{lemma} \label[lemma]{proj im of Gmax}
We have $\GL_2(k)_A=\begin{pmatrix} A & 0 \\ 0 & 1 \end{pmatrix}\SL_2(k)$.  Moreover, $G \leq \GL_2(k)$ contains $\SL_2(k)$ if and only if $\rmP\! G$ contains $\PSL_2(k)$ if and only if $G=\GL_2(k)_A$ where $A=\det G$.  
\end{lemma}

\begin{proof}
If $g \in \GL_2(k)$ has $\det g = a \in A$, then
$g = \begin{pmatrix} a & 0 \\ 0 & 1 \end{pmatrix} g'$
with the first factor in $\begin{pmatrix} A & 0 \\ 0 & 1 \end{pmatrix}$ and the second in $\SL_2(k)$; this proves the first statement.
For the first equivalence, we check directly for $\#k \leq 3$, so we may suppose that $\#k \geq 4$.  Of course if $G$ contains $\SL_2(k)$ then $\rmP\! G$ contains $\PSL_2(k)$; for the converse, if $\rmP\! G \geq \PSL_2(k)$ then $G \cap \SL_2(k) \leq \SL_2(k)$ has index at most $2$ and since $\SL_2(k)$ is equal to its commutator subgroup it has no subgroup of index $2$, thus $G \cap \SL_2(k)=\SL_2(k)$ and so $G \geq \SL_2(k)$.  
For the second equivalence, the containment $G \leq \GL_2(k)_{A}$ is an equality since $\#G = \#A \#\SL_2(k) = \#\GL_2(k)_{A}$.
\end{proof}

The group $\Gal{k}{\F_{\ell}}$ is cyclic, generated by the Frobenius automorphism $x \mapsto x^\ell$; it acts on $\GL_2(k)$ entrywise, giving an injective homomorphism $\Gal{k}{\F_{\ell}} \hookrightarrow \Aut(\GL_2(k))$, and this action descends to $\PGL_2(k)$.
Similarly, the stabilizer of $A$ (in the sense of mapping $A$ to itself, not necessarily fixing its elements pointwise) in $\Gal{k}{\F_{\ell}}$ acts on $\GL_2(k)_A$ and $\PGL_2(k)_A$.
More generally, if $G \leq \GL_2(k)$ has stabilizer $C \leq \Gal{k}{\F_{\ell}}$,
then the natural projection also gives a well-defined homomorphism 
\begin{equation} \label{eqn:semidirA}
\rmP \colon G \rtimes C \to \rmP\!G \rtimes C
\end{equation}
(as $C$ must also stabilize the scalar subgroup of $G$).

Finally, the natural action of $\PGL_2(k)$ on $\PP^1(k)$ extends to an action by $\PGammaL_2(k) := \PGL_2(k) \rtimes \Gal{k}{\F_{\ell}}$ via 
\begin{equation} \label{eqn:PGammaL}
    \left( \begin{pmatrix} a & b\\ c & d \end{pmatrix}, \sigma \right) \cdot (x : y) = (a \sigma(x) + b \sigma(y) : c \sigma(x) + d \sigma(y));
\end{equation}
the action is faithful, so we obtain an injective homomorphism $\PGammaL_2(k) \hookrightarrow S_{n}$ where $n=\#\PP^1(k)=\#k+1$.  (We similarly obtain a permutation representation of $\GammaL_2(k) := \GL_2(k) \rtimes \Gal{k}{\F_{\ell}}$ on $\A^2(k) \smallsetminus \{(0,0)\}$ of degree $(\#k)^2-1$.)

\begin{example} \label[example]{exm:17T7}
Taking $k=\F_{16}$, we have the subgroup 
\[ \SL_2(\F_{16}) \rtimes \Gal{\F_{16}}{\F_4} \leq \SL_2(\F_{16}) \rtimes \Gal{\F_{16}}{ \F_2} \hookrightarrow \Sym(\PP^1(\F_{16})) \cong S_{17} \] 
via the above permutation representation, which as a subgroup of $S_{17}$ is the group \textsf{17T7} (up to conjugacy).

\end{example}

\subsection{Large image} \label[section]{sec:largimage}

We quickly indicate a few statements that together allow us to conclude that a subgroup $G \leq \GL_2(k)$ contains $\SL_2(k)$.

\begin{prop} \label[prop]{prop:maximag}
Let $G \leq \GL_2(k)$ be a subgroup with $\#k \geq 7$.  Then $G$ contains $\SL_2(k)$ if and only if $G$ contains:
\begin{enumroman}
\item a non-scalar split semisimple element with nonzero trace (its characteristic polynomial splits into distinct linear factors in $k$);
\item a nonsplit semisimple element with nonzero trace (its characteristic polynomial is irreducible);
\item an element whose projective order (i.e., order in $\PGL_2(k)$) exceeds $5$; and
\item an element $g$ such that $k=\F_\ell(u(g))$.
\end{enumroman}
\end{prop}

\begin{proof}
The implication ($\Rightarrow$) is direct, so we prove ($\Leftarrow$).  By \Cref{proj im of Gmax}, it is enough to show that $\rmP\!G \geq \PSL_2(k)$.
By Dickson's classification (see e.g.,\ King \cite[Corollaries 2.2--2.3]{King}), the maximal subgroups of $\PGL_2(k)$ are affine, dihedral, exceptional (isomorphic to $S_4$, $A_4$, or $A_5$), or projective (i.e., up to conjugacy, we have $\rmP\! G =\PSL_2(k_0)$ or $\rmP\! G = \PGL_2(k_0)$ for some subfield $k_0 \subseteq k$).  
The stabilizer of a point and the stabilizer of a pair of points are ruled out by (ii), the stabilizer of a pair of imaginary points is ruled out by (i), and the exceptional groups are rules out by (iii).




It follows that $G$ is projective.  But then $u(g) \in k_0$ for all $g \in G$.  By (iv), we conclude $k=k_0$, so $G \geq \PSL_2(k)$.
\end{proof}

We may also work just with traces, as follows.

\begin{prop}[Trace lemma] \label[proposition]{trace lemma} 
Let $G \leq \SL_2(k)$ with $\#k \geq 4$ and $\#k \neq 5$.  Then $G=\SL_2(k)$ if and only if $\tr G = k$.
\end{prop}

\begin{proof}
For $\#k = 4$, we verify the claim with a direct calculation. So we may suppose that $\#k \geq 7$, and apply \Cref{prop:maximag}.
\begin{enumroman}
\item Let $\lambda \in k^\times \smallsetminus \{\pm 1\}$.  Then by hypothesis there exists $g \in G$ such that $\tr(g)=\lambda+\lambda^{-1}$, whence its characteristic polynomial is $x^2-(\lambda+\lambda^{-1})x+1=(x-\lambda)(x-\lambda^{-1})$, so $g$ is split semisimple.  
\item There exists $a \in k^\times$ such that $x^2-ax+1 \in k[x]$ is irreducible, since the map $k^\times \smallsetminus \{\pm 1\} \to k^\times$ given by $\lambda \mapsto \lambda+\lambda^{-1}$ is not surjective: it has fibers of cardinality $2$ and hence image of cardinality at most $(\#k-3)/2 < \#k - 2$.  Any $g \in G$ with $\tr(g)=a$ is therefore nonsplit semisimple.
\item An element $g$ of projective order $\leq 5$ has $\tr g = a \in \{\pm 2,\pm 1,0\}$ or $a^2 \pm a-1=0 \in k$. %
This removes at most $7$ elements from $k$, and when $k=\F_7$ there is no $a \in k$ with $a^2 \pm a-1=0$.  Any $g \in G$ with trace among the remaining elements of~$k$ satisfies (iii).  
\item Let $a \in k^\times$ generate $k^\times$ as an abelian group.  We claim that 
$\F_\ell(a^2)=\F_\ell(a)$.  Indeed, if $\ell=2$ then $\F_\ell(a^2)=\F_\ell(a)$ (as squaring is a Galois automorphism).  If instead $\ell$ is odd, then $\langle a^2 \rangle \leq \F_\ell(a^2)$, so $\#\F_\ell(a^2) \geq (\#k-1)/2 > \#k/\ell \geq \#k_0$ for all subfields $k_0 \subsetneq k$.  Using the claim, any $g \in G$ with $\tr g=a$ will suffice, since then $u(g)=(\tr g)^2=a^2$.
\end{enumroman}
This completes the proof because the conditions in \Cref{prop:maximag} are all satisfied.
\end{proof}

\begin{remark}
\Cref{trace lemma} is false for $\#k=2,3,5$ by the counterexamples $C_3 \leq \SL_2(\F_2)$, $Q_8 \leq \SL_2(\F_3)$, and $\SL_2(\F_3) \hookrightarrow \SL_2(\F_5)$.  However, we can consider an upgraded statement asking for subgroups of $\GL_2(\F_p)$ such the set of characteristic polynomials of elements of $G$ coincides with that of $\GL_2(\F_p)$.  Unfortunately, there is again a counterexample for $p=3$, namely $Q_8 \rtimes C_2 \hookrightarrow \GL_2(\F_3)$; but the result now holds for $\GL_2(\F_5)$ again by a direct calculation.  
\end{remark}

\subsection{Descent} \label[section]{sec:descent}

Now let $F \supseteq F_0$ be a finite Galois extension of number fields inside an algebraic closure $\Qbar$, with absolute Galois group $\Galois_F \colonequals \Gal{\Qbar}{F}$.  Let $k$ be a finite field and let
\begin{equation}
    \rho \colon \Galois_F \to \GL_n(k) 
\end{equation}
be a Galois representation.  Then $\ker \rho$ cuts out a Galois extension $L\colonequals F(\rho) \supseteq F$ with Galois group $\Gal{L}{F}=G \colonequals \img \rho \leq \GL_n(k)$.  Let $S(\rho)$ be the set of (nonzero) primes $\frakp$ of the ring of integers $\Z_F \subseteq F$ above any prime number $p$ that ramifies in $L$.  

The group $\Aut(k)$ acts on $\GL_n(k)$ entrywise, so for $\tau \in \Aut(k)$ we obtain another Galois representation
\begin{equation}
\begin{aligned}
    \tau(\rho) \colon \Galois_F &\to \GL_n(k) \\
    \tau(\rho)(\xi) &= \tau(\rho(\xi)).
\end{aligned}
\end{equation}

There is a second Galois action coming from $G_0 \colonequals \Gal{F}{F_0}$, defined as follows.  Let $\sigma \in \Gal{F}{F_0}$.  Choose a lift $\widetilde{\sigma} \in \Galois_{F_0}$.  We then obtain a new Galois representation defined by
\begin{equation}
\begin{aligned}
    \rho^\sigma \colon \Galois_F &\to \GL_n(k) \\
    \rho^\sigma(\xi) &= \rho(\widetilde{\sigma}^{-1} \xi \widetilde{\sigma})
\end{aligned}
\end{equation}
which is well-defined up to isomorphism, independent of the choice of lift.  In particular, if $\frakp \not \in S(\rho)$ is a prime of $\Z_F$ and $\Frob_\frakp \in \Galois_F$ a Frobenius automorphism at $\frakp$, then 
\begin{equation} 
    \det(1-\rho^\sigma(\Frob_\frakp)T) = \det(1-\rho(\Frob_{\sigma(\frakp)})T) \in k[T],
\end{equation}
well-defined up to conjugacy.  

We organize these according to $\sigma$ as follows.  We say that $\rho$ \defi{descends} to $F_0$ via a group homomorphism 
\begin{equation}
\begin{aligned}
    \tau \colon \Gal{F}{F_0} &\to \Aut(k) \\
    \sigma &\mapsto \tau_\sigma
\end{aligned}
\end{equation}
if there exists an isomorphism $\rho^\sigma \simeq \tau_\sigma(\rho)$ of representations for all $\sigma \in \Gal{F}{F_0}$.  The name is justified by the following theorem.  

\begin{thm}[Galois descent law] \label[theorem]{thm:KF0Gal}
Suppose that $\rho$ descends via $\tau$.  Then the extension $L=F(\rho) \supseteq F_0$ is Galois, fitting in an exact sequence
\begin{equation} \label{eqn:gallff0}
1 \to \Gal{L}{F} \to \Gal{L}{F_0} \to \Gal{F}{F_0} \to 1.   
\end{equation}

Let $G \colonequals \img \rho$, and suppose further that $\tau$ is injective and one of the following holds:
\begin{enumroman}
\item $G=\GL_n(k)$; or 
\item $G=\SL_n(k)$ with $\gcd([F : F_0],n)$ equal to a power of $\ell$ (including 1).  
\end{enumroman}
Then \eqref{eqn:gallff0} splits and $\Gal{L}{F_0} \simeq G \rtimes \Gal{F}{F_0}$, where $\Gal{F}{F_0}$ acts on $G \leq \GL_n(k)$ entrywise via the map $\tau$.
\end{thm}

To begin with the proof of \Cref{thm:KF0Gal}, we record a proof of the first statement.  For all $\sigma \in \Gal{F}{F_0}$, the given isomorphism $\rho^\sigma \simeq \tau_\sigma(\rho)$ in particular implies that $\ker \rho^\sigma = \ker \tau_\sigma(\rho) = \ker \rho$, i.e., $L^\sigma=L$.  But $\ker \rho^\sigma = \widetilde{\sigma} (\ker \rho) \widetilde{\sigma}^{-1}$, so we conclude that $L=N$ is normal over $F_0$.  The exact sequence \eqref{eqn:gallff0} then follows by restriction.  

To go further, we recall the following general setup.

\begin{prop} \label[prop]{prop:oof}
Let $G_0$ and $G$ be groups with $G_0$ finite.  Let
\begin{equation} \label{eqn:extension}
1\to G\to E\to G_0\to 1
\end{equation}
be a short exact sequence.  Then there is an outer action
\begin{equation}
\bar\alpha\colon G_0\to \operatorname{Out}(G).
\end{equation}

Suppose $\bar\alpha$ lifts to
\begin{equation}
\alpha\colon G_0\to \Aut(G).
\end{equation}
Then the extension \eqref{eqn:extension} determines a class
\begin{equation} \label{eqn:ccE}
[c_E]\in H^2(G_0,Z(G)),
\end{equation}
where $G_0$ acts on $Z(G)$ through $\alpha$.  The extension splits, giving $E \simeq G \rtimes_\alpha G_0$, if and only if $[c_E]=1$; in particular, this holds if $Z(G)=\{1\}$.
\end{prop}

\begin{proof}
See Brown \cite[Ch.~IV, \S6]{Brown1982} or Rotman \cite[Ch 7]{Rotman1995}.  
\end{proof}

Referring to \Cref{prop:oof}, we now consider the situation of \Cref{thm:KF0Gal} with $G_0 \colonequals \Gal{F}{F_0}$ and $E \colonequals \Gal{L}{F_0}$.  

\begin{lemma} \label[lemma]{lem:rhotau}
If $\rho$ descends via $\tau$ and $\tau_\sigma(\rho)(\Galois_F) = \rho(\Galois_F)=G$ for all $\sigma \in G_0$, then $\tau \colon G_0 \to \Aut(k)$ lifts to a homomorphism $G_0 \to \Aut(G)$; in particular, this holds if $G=\GL_n(k)$ or $G=\SL_n(k)$.  
\end{lemma}

\begin{proof}
If $\tau_\sigma(\rho)(\Galois_F) = \rho(\Galois_F)=G$, the map $\tau_\sigma \in \Aut(k)$ acting entrywise on $\GL_n(k)$ lifts to an automorphism of $G$ and these combine to a homomorphism $G_0 \to \Aut(G)$.  

If $G=\GL_n(k)$ then equality of images is clear; if $G=\SL_n(k)$, then from $\rho^{\sigma}(\Galois_F)$ being \emph{conjugate} to $\tau_{\sigma}(\rho)$ in $\GL_n(k)$ we may again conclude equality.
\end{proof}

For the second part of \Cref{prop:oof}, when $n=p$ for $G=\SL_n(k)$ already we have $Z(G)=\{1\}$.  

\begin{lemma} \label[lemma]{lem:autg0}
Let $G_0 \leq \Aut(k)$.  Then the following statements hold.
\begin{enumalph}
\item $H^2(G_0,k^\times)=\{1\}$ under the natural action.  
\item Suppose $\gcd(\#G_0,\#\mu_n(k))=1$.  Then $H^2(G_0,\mu_n(k))=1$.  
\end{enumalph}
\end{lemma}

\begin{proof}
Let $k_0 \colonequals k^{G_0}$ be the fixed field, so $G_0 \colonequals \Gal{k}{k_0}$.  Then (using $2$-periodicity of Tate cohomology of finite cyclic groups)
\begin{equation}
H^2(G_0,k^\times)=(k^{\times})^{G_0}/\Nm_{k|k_0}(k^\times) = \{1\}
\end{equation}
handling~(a).  Part~(b) follows immediately from a restriction/corestriction argument.  
\end{proof}

\begin{proof}[Proof of \Cref{thm:KF0Gal}]
We proved the first part right below the statement; for the rest, combine \Cref{prop:oof} with \Cref{lem:rhotau} and \Cref{lem:autg0}, noting that $Z(\SL_n(k))=\mu_n(k)$ and $\#\mu_n(k) \mid n'$ where $n=p^en'$ with $p \nmid n'$.  
\end{proof}


\begin{remark}
The converse of~\Cref{thm:KF0Gal} may not be true, since the condition of being Galois concerns only abstract Galois groups, which may or may not be equivalent as linear representations.  

Moreover, the extra conditions on the group $G$ are necessary for the exact sequence to split.  Indeed, if $F_0 = \Q$, $F = \Q(\sqrt{5})$, and $\rho \colon \mathrm{Gal}_F \to \GL_1(\F_3)$ is the quadratic character factoring through $\Gal{\Q(\zeta_5)}{F}$, then the descent criterion is easily seen to be verified, but the exact sequence does not split.
\end{remark}

\begin{corollary} \label[corollary]{cor:KF0Gal}
Suppose that $\rho$ is semisimple.  If $\tau \colon \Gal{F}{F_0} \to \Aut(k)$ is an injective group homomorphism such that 
\begin{equation} \label{eqn:yupyup}
\det(1-\rho(\Frob_{\sigma(\frakp)})T) = \det(1-\tau_\sigma(\rho)(\Frob_\frakp)T) \in k[T]
\end{equation}
for all primes $\frakp \not\in S(\rho)$, then the extension $L=F(\rho) \supseteq F_0$ is Galois and its Galois group fits in an exact sequence \eqref{eqn:gallff0}, and the Galois group is isomorphic $\Gal{L}{F_0} \simeq G \rtimes \Gal{F}{F_0}$ if $G$ satisfies one of the additional conditions in~\Cref{thm:KF0Gal}.
\end{corollary}

\begin{proof}
Since $\rho$ is semisimple, we just combine~\Cref{thm:KF0Gal} with the Brauer--Nesbitt theorem~\cite[Theorem~30.16]{CurtisReiner} and the Chebotarev density theorem.
\end{proof}

\subsection{Hilbert descent}

We now apply the Galois descent law (\Cref{thm:KF0Gal}) to the situation of a Galois representation attached to a Hilbert modular form, our case of interest.  (The results could also just as easily specialize to any setting where we can attach Galois representations to modular forms.)

Let $F$ be a Galois totally real field of degree $n=[F:\Q]$, and let $f$ be a Hilbert newform over $F$ with level $\frakN$ and paritious weight $(k_i)_i$ with $k_i \geq 2$ for all $i=1,\dots,n$ and Nebentypus character $\chi$.  Let $k_0 \colonequals \max(k_1,\dots,k_n)$.  For $\frakp \nmid \frakN$, let $a_\frakp(f)$ be the Hecke eigenvalue of $f$ at~$\frakp$, and let $K_f \colonequals \Q(\{a_{\frakp}(f)\}_{\frakp})$ be the number field generated by its Hecke eigenvalues (which are themselves algebraic integers in $K_f$).  Let $\frakl$ be a prime of $\Z_{K_f}$ with residue field $\F_\frakl$ and characteristic $\Char \F_\frakl=\ell$.  

\begin{thm} \label[theorem]{thm:HilbertGalrep}
There exists an irreducible Galois reprentation 
\[ \rho_{f,\frakl^\infty} \colon \Galois_F \to \GL_2(K_{f,\frakl}) \]
such that
\begin{equation} \label{eqn:trdetchi}
\begin{aligned}
    \tr(\rho_{f, \frakl^\infty}(\Frob_\frakp)) &= a_\frakp(f) \\
    \det(\rho_{f, \frakl}(\Frob_\frakp)) &= \chi(\frakp)\Nm(\frakp)^{k_0-1} 
\end{aligned}
\end{equation}
for all (nonzero) prime ideals $\frakp$ of $\Z_F$ with $\frakp \nmid \ell\frakN$.
\end{thm}

\begin{proof}
Combine work of Carayol~\cite{Carayol}, Taylor~\cite{Taylor}, and Blasius--Rogawski~\cite{BlasiusRogawski1989}.
\end{proof}

As usual, choosing an integral lattice, reducing modulo $\frakl$, and taking the semisimplification, we obtain a representation
\begin{equation} \label{eqn:redmodell}
\rho^{\textup{ss}}_{f,\frakl} \colon \Galois_F \to \GL_2(\F_\frakl)
\end{equation}
where now \eqref{eqn:trdetchi} holds as congruences modulo $\frakl$.  To simplify notation, we drop the superscript and write just $\rho_{f,\frakl}$.  

In particular, the image of $\rho_{f,\frakl}$ lies in the subgroup $\GL_2(\F_\frakl)_A$ (defined in \eqref{eqn:GL2kA}) where $A \leq \F_\frakl^\times$ is the subgroup generated by $\F_\ell^\times$ and the values of $\chi$ modulo $\frakl$.  %

Let $D_{\frakl} \colonequals \{\sigma \in \Aut(K_f) : \sigma(\frakl) = \frakl\}$ and $I_{\frakl} \colonequals \{\sigma \in D_{\frakl} : \sigma(a) \equiv a \pmod{\frakl} \textup{ for all $a \in K_f$}\}$. (If $K_f$ is Galois, these are the decomposition and inertia groups.)  

\begin{theorem}\label[theorem]{mainthm}
Suppose there is an injective group homomorphism 
\[ \tau \colon \Gal{F}{\Q} \hookrightarrow D_{\frakl}/I_{\frakl} \] 
such that for every prime $\frakp \nmid \ell\frakN$ and for every $\sigma \in \Gal{F}{\Q}$, we have both
\begin{equation} \label{eq:eta sigma on a frakp}
\begin{aligned}
    \tau_\sigma(a_\frakp(f)) &\equiv a_{\sigma(\frakp)}(f) \pmod{\frakl}, \\
    \tau_\sigma(\chi(\frakp)) &\equiv \chi(\sigma(\frakp)) \pmod{\frakl}.
    \end{aligned}
\end{equation}
Then the field $L=F(\rho_{f,\frakl})$ is Galois over $\Q$ with the Galois group fitting in an exact sequence as in~\eqref{eqn:gallff0}.
If additionally $G \colonequals \img \rho_{f,\frakl}$ satisfies the additional conditions in~\Cref{thm:KF0Gal}, then $\Gal{L}{\Q}$ is isomorphic to $G \rtimes \Gal{F}{\Q}$ with $\Gal{F}{\Q}$ acting through $\tau$ on $G \subset \GL_2(\F_{\frakl})$ coefficientwise.
\end{theorem}

\begin{proof}
We apply the form of the Galois descent law (\Cref{thm:KF0Gal}) given in~\Cref{cor:KF0Gal}.  
\end{proof}

\begin{remark} \label[remark]{rmk:only for cyclic}
Since the group $D_\frakl/I_\frakl \hookrightarrow \Gal{\F_\frakl}{\F_\ell}$ is cyclic, \Cref{mainthm} applies only when $F$ is cyclic over $\Q$.  It of course also admits a generalization to the situation where $F \supseteq F_0$ is a cyclic extension of totally real fields, giving a descent to $F_0$ instead of $\Q$.
\end{remark}

As a corollary, we also descend the projective representation.  

\begin{corollary} \label[corollary]{cor:maincor}
Under the hypotheses of \Cref{mainthm} and the additional conditions of \Cref{thm:KF0Gal}, the field $F(\rmP\!\rho_{f,\frakl})$ cut out by the projective representation $\rmP\!\rho_{f,\frakl}$ is Galois over $\Q$ with Galois group $\rmP\!G \rtimes \Gal{F}{\Q}$.
\end{corollary}

\begin{proof}
The projection is well-defined on the semidirect product as in \eqref{eqn:semidirA}.
\end{proof}

\begin{remark} \label[remark]{ref:other17Tx}
The transitive groups \textsf{17T6} and \textsf{17T8} are closely related to \textsf{17T7}: \textsf{17T6} is isomorphic to $\PSL_2(\F_{16})$ and \textsf{17T8} is isomorphic to
\begin{equation*}
    \PSL_2(\F_{16}) \rtimes \Gal{\F_{16}}{\F_2} \simeq \PSL_2(\F_{16}) \rtimes C_4 \, .
\end{equation*}
As explicit polynomials for \textsf{17T6}- \cite{Bosman} and \textsf{17T8}-extensions \cite{JonesRoberts} were already known, one might wonder why the case of \textsf{17T7} remained open.  

The construction used by Bosman in \cite{Bosman} is in fact similar to our construction for \textsf{17T7}, except that it uses classical modular forms instead of Hilbert modular forms. In the classical case, one can use modular symbols to reconstruct the abelian variety instead of Oda's conjecture. 

Jones--Roberts \cite[\S\,13]{JonesRoberts} exhibit two infinite families of \textsf{17T8}-extensions, both arising from Belyi maps defined over $\Q$. Each of these families produces a tower of extensions $L \supseteq K \supseteq \Q(t)$ such that
\begin{equation*} 
\Gal{L}{K} \simeq \PSL_2(\F_{16}) \quad \text{and} \quad \Gal{K}{\Q(t)} \simeq C_4 \, .
\end{equation*}
However, as they remark, for each of these families there is a constant field extension contained in $K$ (containing $\Q(\sqrt{5})$ in each case). This prevents us from obtaining a \textsf{17T7}-extension by specialization of either of these \textsf{17T8} families.
\end{remark}

\begin{remark} \label[remark]{rmk:galalign}
In work of Demb\'el\'e~\cite{Dembele-peq2} and Demb\'el\'e--Greenberg--Voight \cite{DGV}, non-solvable Galois extensions of $\Q$ unramified outside $p = 2,3,5$ were found using Hilbert modular forms over abelian extensions of $\Q$, as above.  Gross explained why in many cases the Galois groups over $\Q$ were semidirect products \cite[\S 1]{DGV}; this observation is encoded in the descent law above.  

See also work of Cunningham--Demb\'el\'e \cite{CunninghamDembele} and Booker--Sijsling--Sutherland--Voight--Yasaki \cite{BSSVY}, who also study the same situation and relate this to abelian varieties of potential $\GL_2$-type. 
\end{remark}

\subsection{Application to the IGP}

We may now put the pieces from the previous subsections together to obtain our application to the Inverse Galois Problem (IGP): we look for Hilbert modular forms $f$ where the mod $\frakl$ image is large (using \cref{sec:largimage}) and $f$ satisfies the descent condition \eqref{eq:eta sigma on a frakp}.

We first focus on the proof of the (ineffective) version of \Cref{thm:17t7main}, realizing \textsf{17T7}, given as in \Cref{exm:17T7}.  Looking at \Cref{thm:intromain}:
\begin{itemize}
\item We need a base field $F$ such that $\Gal{F}{\Q} = \langle \sigma \rangle \simeq C_2$, so we take $F$ a real quadratic field.
\item We need a prime $\frakl$ with residue field $\F_{16}$, so for simplicity we take coefficient field $K_f$ of degree $4$ with $2$ inert.  (There are also fields $K_f$ of degree $>4$ with a prime of residue field~$\F_{16}$, but using such $f$ would make it even harder to compute an explicit polynomial.)
\item We need the image of the determinant to be trivial, so we take trivial Nebentypus character $\chi$.  (Again this is for simplicity; we could take any $\chi$ whose order is a power of $2$.)
\item We must check that \eqref{eq:eta sigma on a frakp} holds; in our case this condition reads
\begin{equation} \label{eqn:asigmafp}
a_{\sigma(\frakp)}(f) \equiv a_{\frakp}(f)^4 \pmod{2}. 
\end{equation}
\item Finally, we need the eigenvalues $a_\frakp(f)$ modulo $2$ to hit every element of $\F_{16}$, so that we can apply \Cref{trace lemma} to deduce that the representation has image~$\SL_2(k)$.
\end{itemize}

To find such forms, we search the database of Hilbert modular forms \cite{DonV} available at the $L$-functions and Modular Forms Database (LMFDB) \cite{LMFDB}.  We restrict to Galois-stable level $\frakN$.  

\begin{remark}
Since $\tau$ is nontrivial, we cannot have $f$ arising as a base change from $\Q$.  We also cannot manufacture such forms from twisted base change.  
To see this, suppose the form comes from twisted base change, say $f=f_0 \tensor \psi$ with $f_0$ from $\Q$ (i.e.\ $a_p(f_0) = a_{\sigma(p)}(f_0)$) and $f$ non-CM.  Then for $\frakp$ a split prime, we have $a_\frakp(f)=a_p(f)\psi(\frakp)$, so the congruence
\[ \tau(a_\frakp(f)) \equiv a_{\sigma(\frakp)}(f) \pmod{2} \]
becomes
\[ \tau(\psi(\frakp))\tau(a_p(f)) \equiv \psi(\sigma(\frakp)) a_p(f) \pmod{2}. \]
Of course if $\psi$ is quadratic, then $\tau(a_p(f)) \equiv a_p(f) \pmod{2}$ so in particular we do not have surjective trace modulo $2$.  Thus $\psi$ must have order at least $3$, so $K_f(\psi)=K_f$ is a CM field.  But then we cannot have trivial Nebentypus character, since then the Hecke field is totally real.
\end{remark}

In principle the congruence modulo $2$ could be proved with a finite computation using Hecke--Sturm bounds \cite{GP17}; however, the relevant bound here would be quite large.  When the congruence lifts to an equality $\tau(f)=f^\sigma$ of eigenforms, with $\tau \in \Gal{K_f}{F}$ the nontrivial involution and $\sigma \in \Gal{F}{\Q}$ the nontrivial element, this can be proved almost instantly from the list of eigenforms by using just the first few Hecke eigenvalues.  

We found 18 Hilbert newforms in the LMFDB with these properties.  We group them according to quadratic twist---since these yield the same mod $2$ Galois representation---and order by (absolute) conductor.  In all cases, it turns out that the desired congruence is in fact an equality.

\begin{remark}
More generally, we would seek to prove an equality $\tau(f)-f^\sigma = 2h$ where $h \in S_2(\frakN)$ has algebraic integer Fourier coefficients, as this implies the desired congruence $\tau(a_\frakp(f)) \equiv a_{\sigma(\frakp)}(f) \pmod{2}$ for all good $\frakp$.  More generally, if there is a number ring $R \subseteq \C$, a Hecke module $M$ over $R$ (i.e., a set of pairwise computing maps $T_\frakp \in \End_R(M)$ for all good $\frakp$) and an isomorphism of Hecke modules $M_\C \to S_2(\frakN)$, then identifying $v \leftrightarrow f$ with $v \in M$ up to rescaling by $R^\times$, it is sufficient to prove that $\tau(v) - v^\sigma \in 2M$, which then becomes a congruence between entries in a pseudobasis for $M$.

This mechanism is particularly convenient working with forms on a definite quaternion order $\calO$, using the Jacquet--Langlands correspondence (indeed, this is one way they can be computed; see Demb\'el\'e--Voight \cite{DV}).  In this case, the vector $v$ is a linear combination of functions on the class set of a definite quaternion order \cite[Chapter 17]{Voight:quat}, having its natural integral structure.  We also verified the congruence (indeed, the equality) this way for all of our forms.  
%
\end{remark}


The forms found are as follows.  

\begin{center}
\begin{tabular}{lll}
Field & Field label & Forms \\
\hline\hline
$\Q(\sqrt{3})$ & \lmfdbnf{2.2.12.1} & \lmfdbshortform{2.2.12.1}{578.1-c}, \lmfdbshortform{2.2.12.1}{578.1-d} ${\phantom{\sqrt{0}^1}}$  \\
$\Q(\sqrt{3})$ & \lmfdbnf{2.2.12.1} & \lmfdbshortform{2.2.12.1}{722.1-i}, \lmfdbshortform{2.2.12.1}{722.1-j},  \lmfdbshortform{2.2.12.1}{722.1-k}, 
\lmfdbshortform{2.2.12.1}{722.1-l} \\
$\Q(\sqrt{2})$ & \lmfdbnf{2.2.8.1} & \lmfdbshortform{2.2.8.1}{2601.1-j}, \lmfdbshortform{2.2.8.1}{2601.1-k} \\
$\Q(\sqrt{2})$ & \lmfdbnf{2.2.8.1} & \lmfdbshortform{2.2.8.1}{2738.1-e}, \lmfdbshortform{2.2.8.1}{2738.1-f} \\
$\Q(\sqrt{3})$ & \lmfdbnf{2.2.12.1} & \lmfdbshortform{2.2.12.1}{1587.1-i}, \lmfdbshortform{2.2.12.1}{1587.1-l}, \lmfdbshortform{2.2.12.1}{1587.1-m}, \lmfdbshortform{2.2.12.1}{1587.1-n} \\
$\Q(\sqrt{6})$ & \lmfdbnf{2.2.24.1} & \lmfdbshortform{2.2.24.1}{726.1-i}, \lmfdbshortform{2.2.24.1}{726.1-j}, \lmfdbshortform{2.2.24.1}{726.1-k}, \lmfdbshortform{2.2.24.1}{726.1-l} \\
\end{tabular}
\end{center}

\begin{theorem} \label[theorem]{thm:17T7ineff}
The group $G=\textup{\textsf{17T7}}$ is a Galois group over $\Q$.
\end{theorem}

\begin{proof}
Applying \Cref{trace lemma} and \Cref{mainthm} to the Hilbert modular forms above, we find at least $6$ different number fields. %
\end{proof}

\begin{remark}
We later also found the Hilbert modular form \lmfdbform{2.2.77.1-99.1-j} as an example of $f^\sigma \equiv \tau(f) \pmod{2}$ but $f^\sigma \neq \tau(f)$.  
\end{remark}

\begin{remark} \label[remark]{rmk:other groups}
While the main focus of this article is $G = \textup{\textsf{17T7}}$, our methods also apply to other groups similarly arising as semidirect products. For example, one can realize $\PSL_2(\F_{64}) \rtimes C_i$ for $i \in \{2, 6\}$ as a Galois group over $\Q$ by exhibiting the forms 
\lmfdbform{2.2.12.1-1250.1-m} and \lmfdbform{6.6.1229312.1-64.1-f}.
\end{remark}

\section{(Re)constructive approach} \label[section]{sec:constructive}

In this section, we show how to explicitly go from a modular form $f$, as in \Cref{sec:existence}, to an explicit polynomial that realizes the desired Galois group, e.g., 17T7.
We first note in \Cref{subsection:Eichler--Shimura construction} that there is an abelian variety $A_f$ associated to $f$ whose $L$-function is related to that of $f$.
In \Cref{SS:PeriodLattice}, we then define a conjectural period lattice attached to $A_f(\C)$.
Next, in \Cref{sect:periods}, we use Oda's conjecture, which can be viewed as an analogue of the BSD formula, to guess the period lattice from the central value of the (twisted) $L$-function of $f$.
Finally, in \Cref{subsec:polynomial}, from this period lattice, we construct a polynomial whose roots correspond to $\frakl$-isogenies from $A_f$, and which has the desired Galois group. 

\subsection{Notation}
In the remainder of the paper, we discuss a constructive method to realize the Galois groups obtained from Hilbert modular forms as in the previous section, and in particular those in \Cref{thm:17T7ineff} realizing \textsf{17T7}.
To accomplish this task, we proceed as outlined in 
\cref{sec:overview}: we (conjecturally) compute periods via twists and construct a moduli point from ratios of these periods, and repeat for the Galois conjugates.  We could then try to reconstruct an abelian variety as a (quotient of a) Jacobian; here, we instead evaluate modular functions to obtain the $2$-isogeny polynomial.

As before, let $f$ be a Hilbert newform of parallel weight 2 and level $\frakN$ over the totally real field $F$, and let $n = [F:\Q]$. Let $K_f \colonequals \Q(\{a_{\frakp}(f)\}_{\frakp})$ be the field generated by its Hecke eigenvalues, and let $g = [K_f:\Q]$. Fix orderings of the embeddings $\sigma_i\colon F \hookrightarrow \R$ where $i=1, \ldots, n$, and $\tau_j\colon K_f \hookrightarrow \C$ where $j = 1, \ldots, g$.

\subsection{Eichler--Shimura construction}
\label[subsection]{subsection:Eichler--Shimura construction}

We begin with the following fundamental conjecture.

\begin{conj}[Eichler--Shimura conjecture] \label[conj]{conj:EichShim}
Let $f$ be a Hilbert newform over $F$ of parallel weight $2$ and level $\frakN$ and Hecke field $K_f$.
Then there exists an abelian variety $A_f$ over $F$ such that
\[ L(A_f,s)=\prod_{j=1}^g L(\tau_j(f),s). \]
More precisely, for every prime $\frakp \nmid \frakN$, we have
\[ L_\frakp(A_f,T) = \prod_{j} L_\frakp(\tau_j(f),T) = \prod_{j} \left(1-\tau_j(a_\frakp(f)) T + \Nm(\frakp) T^2\right) \]
where $a_\frakp(f) \in K_f$ is the $\frakp$-Hecke eigenvalue of $f$.
\end{conj}

\begin{thm} \label[theorem]{thm:EichShim-itsathm}
Suppose that either there exists a prime $\frakq \parallel \frakN$ or that $[F:\Q]$ is odd.  Then \Cref{conj:EichShim} holds.  
\end{thm}

\begin{proof}
Under the given hypothesis, the Eichler--Shimizu--Jacquet--Langlands correspondence holds, and $A_f$ is realized up to isogeny as a quotient of the Jacobian of a Shimura curve \cite[Theorem B]{Zhang}.  For further references, discussion, and examples, see Demb\'el\'e--Voight \cite[Theorem 3.9]{DV}.
\end{proof}

We note that the abelian variety $A_f$ is only well-defined up to isogeny over $F$.  The cases of \Cref{conj:EichShim} missing from \Cref{thm:EichShim-itsathm} are still open, for example when $F$ is a real quadratic field and $\frakN$ is a square.

When $F=\Q$, we can take the Shimura curve to be a modular curve, in which case we can integrate the modular form against a basis of modular symbols to get an analytic realization, giving a big period matrix for $A_f$ (over $\C$).  By contrast, the construction via Shimura curves is a bit oblique: although effective methods are available \cite{GV,VW}, it is still desirable to find an effective way to go more directly from the Hecke eigenvalues (equivalently, the $q$-expansions) of a Hilbert newform to an analytic realization.  

\subsection{Period lattice}
\label[subsection]{SS:PeriodLattice}

In this section, we define a conjectural period lattice attached to a normalized Hilbert newform of parallel weight $2$, following Oda \cite{Oda,Oda2}, Darmon--Logan \cite{DL}, Bertolini--Darmon--Green \cite[\S 7]{Bertolini}, and others, which was made effective for elliptic curves over real quadratic fields by Demb\'el\'e \cite{Dembele}.  

From now on, to simplify notation we abbreviate $K=K_f$. 
Recall that $F$ is a totally real field of degree $n$ and let $\Z_F$ be its ring of integers.
Moreover, let $F_{>0}^\times < F^\times$ be the subgroup of totally positive elements, let 
\begin{equation}
\Fhat \colonequals \prodprime_{\frakp} F_\frakp
\end{equation} 
be the finite adeles of $F$, and let $\Zhat_F \colonequals \prod_{\frakp} \Z_{F,\frakp} \subset \Fhat$ be the profinite completion of $\Z_F$ inside $\Fhat$.  Let $\Gammahat \leq \GL_2(\Zhat_F)$ be a finite index subgroup, let $\h{n}{\pm} \colonequals (\C \smallsetminus \R)^n$, and let 
\begin{equation}
Y(\Gammahat) \colonequals \GL_2(F) \backslash \GL_2(\Fhat) \times \h{n}{\pm} / \Gammahat  
\end{equation}
where $\GL_2(F)$ acts on the first factor by left multiplication and by linear fractional transformations on $\h{n}{\pm}$, with the action on the $i$-th component of $\h{n}{\pm}$ induced by the embedding $\sigma_i$, and where $\Gammahat$ acts by right multiplication on $\GL_2(\Fhat)$.  Then 
\begin{equation}
Y(\Gammahat) = \bigsqcup_{[\frakb]} \Gamma_\frakb \backslash \h{n}{}
\end{equation}
where $\h{n}{} \subseteq \h{n}{\pm}$ is the connected component of $(i,\dots,i)$ (the product of $n$ upper half-planes), the set $[\frakb]$ ranges over the class group $F_{>0}^\times \backslash \Fhat^\times / \det(\Gammahat)$,
and $\Gamma_\frakb$ is idelically conjugate to $\Gamma \colonequals \Gammahat \cap \GL_2(F)_{>0}$, a discrete group acting properly on $\h{n}{}$ \cite[38.7.15]{Voight:quat}.  
Finally, let $X(\Gammahat) \to Y(\Gammahat)$ be a smooth (toroidal) compactification of $Y(\Gammahat)$.  Then $X(\Gammahat)$ is a disjoint union of smooth complex projective varieties of dimension $n$.  

\begin{example}
In our case, we are interested in particular in the following special case: $\widehat{\Gamma} = \widehat{\Gamma}_1(\frakN)$, the standard congruence subgroup such that in the components with $\frakp^e \parallel \frakN$, the matrix is congruent to $\begin{pmatrix} 
1 & * \\ 0 & * 
\end{pmatrix}$
modulo $\frakp^e$.  Then $\det(\Gammahat)=\widehat{\Z}_F^\times$, so the components are indexed by elements $[\frakb] \in \Cl^+ \Z_F$ in the narrow class group of $F$.  
\end{example}

We now define Frobenius elements at infinity as follows.
Let $W_\infty \colonequals \{\pm 1\}^n$.  Write $s_i = (1,\dots,1,-1,1,\dots,1) \in W_\infty$ with $-1$ in the $i$-th place.  Define
$$
\varepsilon_{s_i}(z_1,\dots,z_n) = (z_1,\dots,z_{i-1}, \overline{z_i},z_{i+1},\dots,z_n)
$$
for $z=(z_1,\dots,z_n) \in \h{n}{\pm}$, and extend to $s \in W_\infty$.  Then the action of $W_\infty$ descends to $Y(\Gammahat)$ and then extends to $X(\Gammahat)$ \cite[(1.3)]{Oda2}.

\begin{example}
If there exists $\eta \in \Z_F^\times$ such that $\sgn(\eta)=s$, then we may take $\varepsilon_s((z_i)_i) = (s_i \eta_i z_i)_i$
---this is the star involution in the case of modular curves ($z \mapsto -\overline{z}$, the unit being $-1$).  
\end{example}

Then $W_\infty$ acts on $H^n(X,\Q)$ by pullback, and we get $\varepsilon_s^*$-eigenspaces.  The operators $\varepsilon_s^*$ also commute with
the action of the Hecke operators $T_\frakn$ for ideals $\frakn$ coprime to $\frakN$.

Suppose now that $\Gammahat$ is a standard congruence subgroup and $f$ is a Hilbert newform on $X(\Gammahat)$ with parallel weight $2$.  The eigenspace for the Hecke operators $T_\frakn$ acting on $H^n(X,\Q)$ matching $f$ is a $\Q$-subspace $V_f \subseteq H^n(X,\Q)$ with an action of $K$ such that $\dim_{K} V_f = 2^n$, for example containing
\begin{equation}
\omega_{\tau_j(f)} \colonequals (2\pi i)^n \tau_j(f)(z_1,\dots,z_n)\,\mathrm{d}z_1\,\dots\,\mathrm{d}z_n \in H_{\textup{dR}}^n(X(\Gammahat),\C) \, ,\end{equation}
for any embedding $\tau_j \colon K \hookrightarrow \C$ \cite[(2.1)]{Oda2}.
Moreover, $V_f$ inherits an action of $W_\infty$.  

\begin{thm}
The $K$-vector space $V_f$ can be equipped with a polarized $K$-Hodge structure, with
\begin{equation}
V_f \otimes_\Q \C \simeq \bigoplus_{j} V_f \otimes_{\tau_j} \C 
\end{equation}
such that for all $1 \leq j \leq g$ and all $0 \neq p \leq n$,
\begin{equation}
(V_f \otimes_{\tau_j} \C)^{p,n-p} = \bigoplus_{s} \C \varepsilon_s^*(\omega_{\tau_j(f)})
\end{equation}
where we sum over $s \in W_\infty$ with $p$ plus signs.
\end{thm}

\begin{proof}
See Oda \cite[Construction (3.23) (iii)]{Oda2}.
\end{proof}

Let $\gamma_s \in H_n(X,\Q)$ be a dual basis to $\varepsilon_s^* \omega_{\tau_1(f)}$ where $s$ ranges over $W_{\infty}$. Define
\begin{equation}
\Omega_j^s \colonequals \int_{\gamma_s} \omega_{\tau_j(f)} \in \C.
\end{equation}
for $j = 1, \ldots, g$.
We map
\[ K \hookrightarrow \M_g(\C) \]
by diagonal matrices taking the embeddings $\tau_1,\dots,\tau_g$.
For $s \in \{s_1, \ldots, s_n\}$, let  
\begin{equation} 
V_{f,s} \colonequals K \begin{pmatrix}
\Omega_1^{s} \\ \vdots \\ \Omega_g^{s} 
\end{pmatrix} 
\oplus K \begin{pmatrix}
\Omega_1^{+} \\ \vdots \\ \Omega_g^{+} 
\end{pmatrix} \subsetneq \C^g
\end{equation} 
where we abbreviate $+=(+1,\dots,+1)$.

\begin{conj}[{\cite[Main Conjecture $\textup{A}^{\textup{split}}$, p.\,xii]{Oda}}] \label[conj]{conj:odaconj}
For any choice of lattice $\Lambda \subset V_{f,s}$, we have 
\[ \C^g/\Lambda \sim A_f(\C) \]
for $A_f$ as in \Cref{conj:EichShim}, under the corresponding embedding $\sigma \colon F \injects \C$, i.e., if $s = s_i$, then $\sigma = \sigma_i$.
\end{conj}

The choice of lattice is rather unclear at this point: we may start with 
\begin{equation} \label{eqn:lattice}
\Lambda_s(\fraka,\frakb) = \fraka \begin{pmatrix}
\Omega_1^{s} \\ \vdots \\ \Omega_g^{s} 
\end{pmatrix} 
\oplus \frakb \begin{pmatrix}
\Omega_1^{+} \\ \vdots \\ \Omega_g^{+} 
\end{pmatrix}
\end{equation}
with $\fraka,\frakb$ fractional ideals of $K$.  On the lattice $\fraka \oplus \frakb$ and $c \in K_{>0}^\times$, we define the alternating $\Z$-linear pairing
\begin{equation}
\begin{aligned}\label{eqn:lattice-pairing}
E_c \colon (\mathfrak{a} \oplus \mathfrak{b}) \times (\mathfrak{a} \oplus \mathfrak{b}) &\to \mathbb{Q} \\
        (a_1, b_1), (a_2, b_2) &\mapsto \Tr_{K|\mathbb{Q}}(c(a_1 b_2 - a_2 b_1)).
\end{aligned}
\end{equation}
The pairing can also be considered as a pairing on $\Lambda_s(\fraka,\frakb)$.  
For this pairing to induce a principal polarization, we require 
\begin{equation} 
c\fraka\frakb=\Z_{K}^{\sharp},
\end{equation}
where $\Z_{K}^{\sharp} \colonequals \{a \in K : \Tr_{H|\Q}(a\Z_{K}) \subseteq \Z\}$ is the trace dual (i.e., the codifferent) of $\Z_{K}$. (Cf.~\cite[Corollary 2.10]{Goren}.)
When $\Cl^+(\Z_{K})$ is trivial, we take just $\Lambda_s(\Z_{K},\Z_{K})$ with $c$ a totally positive generator of the codifferent.

Choosing a symplectic basis for $\Lambda_s(\Z_{K},\Z_{K})$ with respect to the pairing $E_c$ \eqref{eqn:lattice-pairing}, one obtains a big period matrix $\Pi_s \in \M_{g,2g}(\C)$.
The \defi{small period matrix} $Z \in \hSiegel{g}$ is the $g \times g$ symmetric matrix with totally positive definite imaginary part defined by the property
  \begin{equation}\label{eqn:small-period-matrix}
  \Pi_s \sim \begin{pmatrix} Z & 1_g \end{pmatrix}.
  \end{equation}

  In similar fashion we define
\begin{equation} \label{eqn:H_moduli}
z_{f,s} \colonequals \left(\frac{\Omega_1 ^{s}}{\Omega_1 ^{+}}, \ldots, \frac{\Omega_g ^{s}}{\Omega_g ^{+}}\right)  \in \h{g}{}.
\end{equation}
With the choices made above, this is the same abelian variety with RM by $\Z_K$ as a point in one component of the Hilbert moduli space.
Of course, this moduli point is only unique up to the natural action of $\SL_2(\Z_K)$.  

\subsection{Periods and $L$-values}\label[subsection]{sect:periods}

Recall that each embedding $\tau_j$ gives rise to an $L$-function
\[
    L(\tau_j(f),s) =  \prod_{\frakp \mid \frakN} \bigl(1 - \tau_j(a_\frakp) \Nm(\frakp)^{-s}\bigr)^{-1} \prod_{\frakp \nmid \frakN} \bigl(1 - \tau_j(a_\frakp) \Nm(\frakp)^{-s} + \Nm(\frakp)^{1-2s} \bigr)^{-1}
\]

Let $\frakc \subseteq \Z_F$ be a nonzero ideal.  Let $F_{\frakc\infty}$ be the Hilbert class field of $F$ of conductor $\frakc\infty$, and let 
\begin{equation}
\chi \colon \Gal{F_{\frakc\infty}}{F} \to \C^\times 
\end{equation}
be a (narrow ray class) character.
By class field theory, $\chi$ corresponds also to a (finite order) Hecke character of modulus $\frakc$.  Associated to $\chi$ is its finite part
\begin{equation}
\chi_0 \colon (\Z_F/\frakc)^\times \to \C^\times
\end{equation}
and an infinite part \defi{sign} 
\begin{equation}
\chi_{\infty} \colon \{\pm 1\}^n \to \{\pm 1\},
\end{equation}
satisfying the compatibility
\begin{equation}
\chi(a\Z_F)^{-1} = \chi_0(a)\chi_{\infty}(\sgn(a)) 
\end{equation}
for all $a \in \Z_F$ coprime to $\frakc$, where $\sgn \colon F^\times \to \{\pm 1\}^n$ records the signs under the real embeddings of $F$.  In the notation above, we have $\chi_{\infty} \in W_\infty$.  

Denote by $\Lchi{s}$ %
the twist of this $L$-function by the Hecke character $\chi$. The Euler factors of this twisted $L$-function at the primes $\frakp$ not dividing $\frakc + \frakN$ are
\begin{equation}
    1 - \tau_j(a_\frakp \chi(\frakp)) \Nm(\frakp)^{-s} + \tau_j(\chi(\frakp))^2 \Nm(\frakp)^{1-2s}.
\end{equation}
Moreover, it has an analytic continuation to the whole complex plane, and its completed $L$-function satisfies a functional equation.

The following theorem, originally stated by Oda in \cite[Prop.\ 16.3]{Oda} for $F$ of narrow class number one, relates the twisted periods with the special values of certain twisted \hbox{$L$-functions} associated to the Hilbert modular form.

 \begin{theorem}[\textrm{\cite[Theorem VI.7.5]{VanderGeer}}]\label[theorem]{conj:oda}
     Let $\chi\colon (\Z_F/\frakc)^\times \to \C^\times$ be a quadratic character of sign $s \in W_\infty$.
     There exists $\alpha_{\chi} \in K$ such that for all $j = 1, \ldots, g$
    \begin{align*}
      \tau_j(\alpha_{\chi}) \Omega^{s}_j &= -4\pi^2 \sqrt{\disc(F)} G(\overline{\chi}) \Lchi{1},
    \end{align*}
    where $G(\chi)$ is the Gauss sum of $\chi$.
 \end{theorem}
  Based on the Birch and Swinnerton-Dyer conjecture, it is conjectured that $\alpha_\chi$ actually lies in $\Z_{K}$ for $\mathrm{cond}(\chi) \gg 0$; see \cite[\S\,7]{Bertolini} or \cite[Conjecture~3.3]{Dembele} (when $F$ has narrow class number 1).

While all the terms on the right hand side of the equality can be computed, the same is not immediately true for $\alpha_{\chi}$. In comparison with the Birch and Swinnerton-Dyer conjecture, the $\alpha_\chi$ correspond to some of the invariants like the Tamagawa numbers that can vary for different characters $\chi$ with the same sign $s$. One can use multiple characters $\chi$ for each sign $s$ and use a lattice based method to determine a likely value for $\alpha_{\chi}$. This trick, also called Cremona's trick, is described in \cite[Remark~5.2]{Dembele} and has its origin in \cite[Section~2.11]{Cremona}.

\subsection{\texorpdfstring{$\frakl$}{frakl}-isogeny polynomial} \label[subsection]{subsec:polynomial}

Let $A$ be an abelian variety defined over $\C$ with real multiplication by an order $\mathcal{O}$ in a number field $K$.
Let $\frakl$ be an ideal of $\mathcal{O}$.
For simplicity, we assume $\frakl$ is principal.
We now define a polynomial generating the $\frakl$-isogeny field of the corresponding real multiplication abelian variety.
We work with $\frakl$-isogenies rather than $\frakl$-torsion points because the splitting field of the resulting polynomial cuts out the projective representation $\rmP\!\rho_{f,\frakl}$ (\Cref{cor:maincor}), which is what we need for our Galois realization.

\newcommand{\newG}{e}
Let $\newG$ 
be a Hilbert modular form of weight $k$ of trivial level (below we will use the restriction of a Siegel Eisenstein series), and $z \in \h{g}{}$ be a point corresponding to the $A$ in the Hilbert moduli space, see \eqref{eqn:H_moduli}.

For $\gamma = \begin{pmatrix}a &b\\ c& d\end{pmatrix} \in \GL_2(\mathcal{O})$, we define the factor of automorphy
\begin{equation*}
    j(\gamma, z) \colonequals \prod_{\iota : K \hookrightarrow \R}\frac{\left( \iota(c) + \iota(d)z \right)^k}{\det(\iota(\gamma))^{k}}.
\end{equation*}
Thus, we can define the \defi{$\frakl$-isogeny polynomial}
\begin{equation*}
    T_{\frakl} (x) \colonequals \prod_{\gamma} \left(x - j(\gamma, z)^{-1} \frac{\newG(\gamma z)}{\newG(z)}\right) \in \C[x],
\end{equation*}
where $\gamma$ ranges over some elements of $M_2(\mathcal{O}) \cap \GL_2(K)$ whose determinant generates $\frakl$, such that $z' = \gamma z$ ranges over points corresponding to all abelian varieties $\frakl$-isogenous to $A$. (If $\frakl$ were actually not principal, then this does work, and can be salvaged by doing it adelically.)

This $\frakl$-isogeny polynomial simultaneously generalizes the classical modular polynomial in two distinct directions. First, rather than relying on the classical $j$-invariant, it allows the use of an arbitrary modular function (constructed as a ratio of modular forms) with rational Fourier coefficients. Second, it naturally extends the theory to higher-dimensional abelian varieties with real multiplication, replacing evaluations on the classical $j$-line with evaluations on the Hilbert moduli space.

For practical purposes, it is convenient to use a modular form $\newG$ that can be expressed in terms of the Riemann theta function through the natural embedding $\h{g}{} \hookrightarrow \hSiegel{g}$, see \eqref{eqn:small-period-matrix}.
For $a, b \in \frac{1}{2}\Z^g/\Z^g$, the Riemann theta function with characteristics $a,b$ is defined as
\begin{equation}
\begin{aligned}
    \thetacha{}{a}{b}: \C^g \times \hSiegel{g} &\to \C\\
    (w, Z) &\mapsto \sum_{n\in {\Z}^g} \exp\left( \pi i (n+a)^t Z (n+a) +2 \pi i (n+a)^t(w+b)\right)\,.
\end{aligned}
\end{equation}
For our computations we pick $\newG$ to be the restriction to $\h{g}{}$ of the Siegel Eisenstein series of weight $4$, which can be expressed as the sum of the eighth powers of the theta constants:
\begin{equation}
    E_4(Z) \colonequals \sum_{a,b} \thetacha{}{a}{b}(0, Z)^8 \, .
\end{equation}
(See \cite[p.~405]{Igusa64}.)

Taking $Z$ to be the small period matrix \eqref{eqn:small-period-matrix}, the $\frakl$-isogeny polynomial $T_\frakl(x)$, which can be expressed in terms of the Siegel Eisenstein series evaluated at $Z$, is defined over $F$ as a consequence of \Cref{conj:odaconj}.

\section{The inverse Galois problem for 17T7} \label[section]{sec:17T7}

In this section, we explain in detail the calculation that gives the \textsf{17T7} polynomial in \Cref{mainthm}. The code used to perform these computations is available at~\cite{GitHubRepo17T7} and~\cite{GitHubRepoEichlerShimuraHMF}.

\subsection{Computing the small period matrix}\label[subsection]{subsect:small-period-matrix-computation}
Let $F \colonequals \Q(\sqrt{3})$ and let $\Z_F =\Z[\sqrt{3}]$ be its ring of integers (of discriminant~$12$).  Then $\Z_F$ has class number~$1$ but narrow class number~$2$, with the narrow class group $\Cl^+ \Z_F$ generated by the unique prime $(1+\sqrt{3})$ above $2$.  The narrow Hilbert class field is $F(\sqrt{-1})=\Q(\zeta_{12})$.  

Let $f$ be the Hilbert modular form over $F$ with LMFDB label \lmfdbform{2.2.12.1-578.1-c}. Then $f$ has level $\frakN = 17(1 + \sqrt{3})$, trivial Nebentypus character, and Hecke eigenvalue field $K = K_f = \Q(\nu)$ with LMFDB label \lmfdbnf{4.4.725.1} and defining polynomial
\begin{equation} \label{eqn:heckefield}
x^4 - x^3 - 3x^2 + x + 1\, .
\end{equation}
Nearby is the form $f'$ with label \lmfdbform{2.2.12.1-578.1-d}, which is the quadratic twist of $f$ by the nontrivial character of the narrow class group.  

\begin{remark}
We could equally well compute in this section and the next with the form $f'$; we just chose to work alphabetically.   
\end{remark}

Using \Magma{} (see Demb\'el\'e--Voight \cite{DV} for a description of the algorithms), 
we compute Hecke eigenvalues $a_{\frakp}$ of $f$ for all prime ideals $\frakp$ of $F$ with $\Nm(\frakp) < 80\,000$. %
We form the truncation of the $L$-function using these $a_\frakp$ and, using Hecke characters $\chi$ with conductors up to 25, compute twisted periods as described in \Cref{sect:periods} using \cite{Dokchitser-lfun}.
In fact, to get more precision for our computation with the $a_{\frakp}\!$'s that we computed, we use the fact that $\Omega_j^{++}\Omega_j^{--} + \Omega_j^{+-}\Omega_j^{-+} = 0$, which follows from \cite[Theorem 4.4]{Oda}.
This yields RM moduli points with $80$ decimal digits of precision
\begin{align*}
z_{+-} &\approx 
(2.7829 i, 0.75416 i, 1.4277 i, 5.0448 i)\\
z_{-+} &\approx (0.75416 i, 2.7829 i, 5.0448 i, 1.4277 i)
\end{align*}
as in \eqref{eqn:H_moduli}.
Note that $z_{+-}$ and $z_{-+}$ are, up to precision, related by the double transposition $(1\ 2)(3\ 4)$. This is because in fact our abelian fourfold descends to $\Q$ (the two conjugates over $F$ are isomorphic), whence it suffices to simply consider the first moduli point $z \colonequals z_{+-}$.

We compute that the different ideal $\frakD_{K} = (\Z_K^\sharp)^{-1} = (d)$ of $K$ (cf.\ \Cref{SS:PeriodLattice}) is narrowly principal with generator $d \colonequals -2 \nu^3 + 4\nu^2 + 3\nu + 2$.  Since the different is narrowly principal, the abelian fourfold $\C^g \,/\, \Lambda$ with $\Lambda \coloneq \Lambda_{+-}(\Z_{K}, \Z_{K})$ as in \eqref{eqn:lattice} is principally polarizable; see \Cref{SS:PeriodLattice}.
We note that $\Lambda$ is equipped with the pairing
\[
    \begin{aligned}
        E_{d^{-1}} \colon (\Z_{K} \oplus \Z_{K}) \times (\Z_{K} \oplus \Z_{K}) &\to \mathbb{Q} \\
        (a_1, b_1), (a_2, b_2) &\mapsto \operatorname{Tr}_{K/\mathbb{Q}}(d^{-1}(a_1 b_2 - a_2 b_1)) \, ,
    \end{aligned}
\]
as in \eqref{eqn:lattice-pairing}.

The periods \(\Omega_j^s\) obtained from \Cref{conj:oda} are purely imaginary or purely real; thus the same holds for \(z_s\). Consequently, the complex torus constructed above may be off by a \hbox{\(2\)-isogeny}. 
Indeed, consider any lattice $L \subset \C^g$ invariant under complex conjugating, and let $L'$ be the sublattice generated by the real and totally imaginary parts of $L$.
Then, for every $\lambda \in L$, both $\lambda+\overline{\lambda}$ and $\lambda-\overline{\lambda}$ lie in $L'$. Hence $2\lambda\in L$, so $L/L'$ is killed by $2$.
Equivalently, the natural map
\[
\C^g/L' \longrightarrow \C^g/L
\]
has kernel contained in the \(2\)-torsion. Thus we search over all \(2^g+1=17\) abelian varieties that are \hbox{\(2\)-isogenous} to \(\C^g/\Lambda\) and have RM by \(\Z_K\).

For this particular example, the knowledge that our desired abelian $4$-fold is a Jacobian (see \Cref{sect:Shimura} below) simplifies this step of the calculation.
We find a unique $2$-neighbor $z' \in \h{g}{}$ such that the value of the Schottky modular form at its period matrix has absolute value $< 10^{-56}$. Thus this is the only likely Jacobian among the $2$-neighbors of $z$.

\subsection{Finding the 2-isogeny polynomial}

We compute the $2$-isogeny polynomial for $Z$ as in \cref{subsec:polynomial}, evaluating theta functions using the fast code of Elkies--Kieffer \cite{EK} available in FLINT \cite[\texttt{acb\_theta}]{Flint}.  We find the polynomial
\begin{align*}
    T(x) &\colonequals x^{17} -581020.41645\ldots x^{16} -54729032212.54644\ldots x^{15}\\
    &\quad- 2958404450460894.75024\ldots x^{14} + \cdots
\end{align*}

We are able to recognize the coefficients of $x^{16}$ and $x^{15}$ (which are known to the highest precision) as rational numbers with denominators $D \colonequals 267075169$ and $D^2$, respectively.
Replacing $T(x)$ by $T_D \colonequals D^{17} T(x/D)$ in order to clear this denominator, we are then able to recognize the coefficients of $x^{16}, x^{15}, x^{14},$ and $x^{13}$ as integers $a_1, a_2, a_3,$ and $a_4$.
In order to obtain higher precision approximations for the rest of the coefficients of $T_D$, we use Newton's method on the function that takes a point $z \in \h{4}{}$ to the coefficients of $x^{16}, x^{15}, x^{14},$ and $x^{13}$ of its 2-isogeny polynomial.
This produces a normalized $2$-isogeny polynomial $T_D$ with coefficients that are easily recognized as integers: 
\begin{equation}
\begin{aligned}
T(x) &= x^{17} - 155176125916688 x^{16} - 3903775123456327337126372744 x^{15} - \dots \\
&\qquad\qquad  - 15370284691667761315579594335774216542251094826\cdots 14304
\end{aligned}
\end{equation}
where the constant term has 204 decimal digits.

Applying the \textsf{PARI/GP} \cite{pari} command \texttt{polredabs} \cite{polredabs},
we find the simplified polynomial given in \eqref{eqn:fx17} 
that defines an isomorphic number field.  We verify that this polynomial has Galois group \textsf{17T7} using the \Magma{} commands \texttt{GaloisGroup} and \texttt{GaloisProof}; this uses the method of relative invariants due to Stauduhar~\cite{Stauduhar}, and the implementation is elaborated upon in Fieker--Kl\"uners~\cite{FK}.  The ring of integers of the number field $L$ defined by this polynomial has discriminant $2^{44} \cdot 3^6 \cdot 17^8$ and has now been included in the LMFDB with label \lmfdbnf{17.1.89462021750334834736103424.1}.

The total CPU time was dominated by the computation of the eigenvalues of the Hilbert modular form---we did not keep a precise count of this time (we computed more than we needed), but it was on the order of a few CPU years.

\begin{remark}
    In general, the polynomial $T(x)$ will have coefficients in $F$ and not necessarily in $\Q$.
    In that case, we can recognize its coefficients by considering all embeddings of $T$ into $\C$ simultaneously.
    Then the extension as in \Cref{mainthm} can be found by taking the splitting field of $T$ over $F$, and then taking its normal closure over $\Q$.
\end{remark}

\begin{remark}
Our number field $L$ has class number $3$, a fact which may seem quite remarkable.  Using \textsf{Magma}, we can explicitly construct this Hilbert class field.  In fact, there is some structure behind this unramified extension, as follows.  

Let $S$ be the stabilizer of an element in $\PP^1(\F_{16})$ under the action of \textsf{17T7} by linear fractional transformations, and let $L$ be the fixed field of $S$.  (Its commutator subgroup $S' \trianglelefteq S$ in fact has index $6$.)  
There is an index $3$ subgroup coming from the natural further scaling action on the fixed vector: it is the subgroup  $\F_4^\times \leq \F_{16}^\times$ stabilized by the order $2$ subgroup of $\Gal{\F_{16}}{\F_2}$.  In this way, we obtain a cyclic extension $L' \supseteq L$ of degree $3$ for all \textsf{17T7} extensions.  In our case, the abelian variety is semistable, so the inertia groups at $2$ and $17$ act through a unipotent subgroup and hence the extension $L' \supseteq L$ is unramified.  (See also the proof of \Cref{prop:modularity} below.)
\end{remark}

\subsection{Relation to Shimura curves}\label[subsection]{sect:Shimura}
Although it is not necessary for our method, it is natural to wonder why the abelian fourfold we have constructed numerically seems to be (isogenous to) a Jacobian.  Indeed, in this special case, we have a very exceptional situation and can prove that this is the case.  

First, we set up the notation and do some preliminary calculations.  Let $\frakp_2= (1+\sqrt{3})$ be the unique prime above $2$ in $\Z_F$.  Then $\frakp_2$ generates the narrow class group $\Cl^+ \Z_F \simeq \Z/2\Z$.  Let $B$ be the quaternion algebra over $F=\Q(\sqrt{3})$ ramified at $\frakp_2$ and at one of the two real places.  

Let $\calO$ be an Eichler order of prime level $(17)$.  We claim that $\calO$ is unique up to conjugation in $B^\times$ (i.e., the type set of $\calO$ is trivial).  Indeed, $\calO$ is hereditary \cite[\S 23.3]{Voight:quat} so its idelic normalizer has \cite[Corollary 23.3.14]{Voight:quat}
\[ N_{\widehat{B}^\times}(\widehat{\calO})/(\widehat{F}^\times \widehat{\calO}^\times) = \langle \varpi_2,\varpi_{17} \rangle \simeq \Z/2\Z \times \Z/2\Z, \]
generated by elements $\varpi_2$ supported at $F_{\frakp}$ and $\varpi_{17}$ supported at $F_{17}$ and whose reduced norms are uniformizers (so $\nrd(\varpi_2)=1+\sqrt{3}$ and $\nrd(\varpi_{17})=17$).  Then, as a consequence of strong approximation \cite[Corollary 28.5.10]{Voight:quat} we compute that the type set of $\calO$ is the quotient of $\Cl^+ \Z_F$ by the ideals $\frakp_2$ and $(17)$, so is indeed trivial.  

Now since $B$ is split at the other real place, it yields an embedding $\iota_\infty \colon B \hookrightarrow \M_2(\R)$ unique up to conjugation by $\GL_2(\R)$.  Let $\calO_{>0}^\times$ be the group of units of $\calO$ whose reduced norm is positive at both real places (they are automatically positive at the ramified real place).  Then under $\iota_\infty$, the group $\rmP\!\calO_{>0}^\times = \calO_{>0}^\times/\Z_F^\times$ is a discrete group acting properly on the upper half-plane $\h{}{}$, and the quotient $X^+(\frakp_2; 17) \colonequals \rmP\!\iota_\infty(\calO_{>0}^\times) \backslash \h{}{}$ is a Shimura curve of genus $11$ with signature $(11;2,2,3,3,12,12;0)$ (compact with six elliptic points of the given orders) \cite{Voi:funddom,Rickards}.  Finally, since $(17)$ is narrowly principal there is an Atkin--Lehner involution $w_{17} \in N_{B^\times}(\calO)_{>0}$, and the further quotient $X^+(\frakp_2;17)/\langle w_{17} \rangle$ has genus $4$ (more precisely, signature $(4;2,2,2,2,2,2,2,2,2,3,12;0)$).  

\begin{prop} \label[prop]{prop:descendtoQQ}
There exists a smooth, projective, geometrically integral curve $X'$ defined over $\Q$ with the property that $X'(\C) \simeq X^+(\frakp_2;17)/\langle w_{17} \rangle$ and $\Jac X'_F$ is isogenous to $A_{f'}$ where $f'$ is the Hilbert modular form with label \lmfdbform{2.2.12.1-578.1-d}. 
\end{prop}

\begin{proof}
The idelic Shimura curve \cite[\S 38.7]{Voight:quat} attached to $\widehat{\calO}^\times$, namely
\[ B_{>0}^\times \, \backslash \, (\widehat{B}^\times \times \h{}{}) \, / \, \widehat{\calO}^\times \]
has two components, indexed by $\Cl^+ \Z_F$ \cite[(38.7.14)]{Voight:quat}.  Its canonical model (due to Shimura \cite{Shimura} and Deligne \cite{Deligne}) is defined over the reflex field, which is $F$; and the components are defined over the narrow class field of $F$, namely $F(\sqrt{-1})$.  The Atkin--Lehner involution $w_{17}$ (preserving components) is defined over $F$.  

This matches the associated calculation of Hilbert modular forms of parallel weight $2$.  The space of Hilbert modular forms of level $\frakN$ which are new at $\frakp_2$ are those of levels $\frakp_2$ and $\frakN=17\frakp_2$, but there are \href{https://www.lmfdb.org/ModularForm/GL2/TotallyReal/?field_label=2.2.12.1&level_norm=2}{no newforms of level norm $2$} so all are \href{https://www.lmfdb.org/ModularForm/GL2/TotallyReal/?field_label=2.2.12.1&level_norm=578}{newforms of level norm $578$}. 
The total dimension of this space is $1+1+4+4+6+6=22=2\cdot 11$.
Moreover, the forms with Atkin--Lehner eigenvalue $1$ for $(17)$ form a space of dimension $4+4=8=2\cdot 4$, so the Jacobian of the quotient of the above idelic Shimura curve by the Atkin--Lehner involution $w_{17}$ is isogenous (over $F$) to the product of the abelian varieties attached to $f'$, the quadratic twist of $f$ (given by \lmfdbform{2.2.12.1-578.1-c} and computed in the previous section) by the nontrivial narrow class character.

By a theorem of Doi--Naganuma \cite[Corollary, p.~449]{DN}, since the type number of $\calO$ is $1$ (their ``narrow sense'', see 1.4 in loc.\ cit.), the field of moduli of either component is~$\Q$.  Finally, the elliptic point of order $12$ is unique, so the provided isomorphisms among the conjugates over $F(\sqrt{-1})$ must be pointed.  We claim moreover that the Atkin--Lehner involution $w_{17}$ also has field of moduli $\Q$: indeed, the subgroup of geometric automorphisms of a curve of genus $\geq 2$ fixing a point is a finite cyclic group (in characteristic $0$), so an involution is uniquely determined by its set of fixed points when nonempty.  We confirm from the signature that the involution $w_{17}$ has fixed (CM) points (or more simply, a fixed-point free involution of a curve of genus $11$ has quotient of genus $6$ by Riemann--Hurwitz).  It follows then by pointed descent \cite[Theorem A]{SV} that the curve, the point, and the quotient map descend canonically to $\Q$!  

This applies to both components, so we obtain two curves of genus $4$ over $\Q$.  Upon extension to $F$, their Jacobians give the two abelian fourfolds under consideration, so one corresponds to $A_f$ over $F$.  We can isolate this component directly: there is still an Atkin--Lehner involution attached to $\frakp_2$ defined over $F$ defined on the idelic Shimura curve: it has degree~$1$ (since $\frakp_2$ is ramified) and interchanges the two components.  The quotient by this automorphism picks out a component according to its eigenvalue for this involution, opposite to that of the Hilbert modular form.
\end{proof}




\begin{remark}
A moduli-theoretic proof of the descent in \Cref{prop:descendtoQQ} should also be possible: the data that defines the moduli problem is defined over $F$; in particular the Atkin--Lehner involutions are defined over~$F$, and we find that there is an isomorphism to the moduli problem which is conjugate under the nontrivial element of $\Gal{F}{\Q}$. This is another way to view the aforementioned result of Doi--Naganuma.  This ensures that the field of moduli is $\Q$, and pointed descent then gives field of definition $\Q$. \end{remark}

In \cite{Bouchet}, Bouchet presents a collection of invariants that classify genus $4$ curves over an algebraically closed field; in \cite{bouchet2} he shows, given such a collection of invariants, how to produce a genus $4$ curve with these invariants.
Starting with the small period matrix $Z$ \eqref{eqn:small-period-matrix}, we first apply the methods of Hanselman--Pieper--Schiavone \cite{HPS} to produce approximate equations over $\C$ of a genus $4$ curve corresponding to $Z$.
We then compute numerical approximations of the Bouchet invariants of this curve, recognize them as rational numbers, and then reconstruct a genus $4$ curve with these invariants using \cite{bouchet2}.
We obtain the curve $X_0$ (the notation chosen to indicate descent to $\Q$) given by
\begin{equation} \label{eqn:genus4curve}
\begin{aligned}
     0 &= -8 x^2 + 8 x y + 17 y^2 - 34 x z - 2 y z - 28 z^2 - 10 x w - 9 y w - 18 z w + 2 w^2,\\
     0 &= 4 x^3 - 6 x^2 y - 6 x y^2 + 12 x^2 z + 6 x y z + 24 y^2 z -12 x z^2 - 24 z^3 + 2 x^2 w + 7 x y w\\
     &\qquad + 4 y^2 w + 4 x z w - 13 y z w - 8 z^2 w - 20 x w^2 - 3 z w^2 - 12 w^3
     \end{aligned}
 \end{equation}
inside the projective space $\PP^3$ with coordinates $x,y,z,w$.

\begin{remark}
If we had computed in the previous two sections with $f'$ instead of $f$, since the associated abelian fourfolds being isomorphic over $\C$ we would have obtained the same invariants and hence the same reconstruction \eqref{eqn:genus4curve}.  
\end{remark}

A computation with discriminants shows that this model has good reduction away from the primes $2$, $3$, $7$, and $17$.  Projection from the point $(-12:2:4:3) \in X_0(\Q)$ and a change of variables to minimize yields the singular affine plane model as in \eqref{eqn:singmod}.  Projection from the point $(-2:-2:3:1)$ gives the model
\begin{equation}
\begin{aligned}
&20 x^5 + 34 x^4 y + 172 x^4 - 6 x^3 y^2 + 60 x^3 y + 576 x^3 - 61 x^2 y^3 - 
    606 x^2 y^2 + 204 x^2 y  \\
    &\qquad + 776 x^2 - 75 x y^4 - 822 x y^3 - 972 x y^2 + 
    168 x y + 384 x - 34 y^5 \\
    &\qquad - 338 y^4 - 688 y^3 - 448 y^2 - 32 y + 32 = 0
    \end{aligned}
    \end{equation}
which reduces modulo $7$ to a curve of genus $4$, giving a model for $X_0$ with good reduction at $7$.  

\begin{prop} \label[prop]{prop:modularity}
Let $A_0 \colonequals \Jac(X_0)$ be the Jacobian of $X_0$.  Then $A_{0,F}$ is isogenous to $A_{f'}$ over $F$.
\end{prop}


Our computations are performed in \textsf{Magma}.

\begin{proof}
We certify using the methods of Costa--Mascot--Sijsling--Voight \cite{endos} that indeed $\End A_0 \simeq \Z[(1+\sqrt{5})/2]$  (over $\Q$) and $\End A_0^{\textup{al}}$ is isomorphic with the ring of integers $\Z_{K}$ of the quartic field $K$, and the endomorphisms are defined over $F=\Q(\sqrt{3})$ (i.e.\ $\End A_0^{\textup{al}} = \End (A_0)_{\Q(\sqrt3)}$).
The endomorphism ring is certified by exhibiting an explicit correspondence, i.e., by representing endomorphisms by their graphs in $\Sym^4 X_0 \times \Sym^4 X_0$, via the birational map $A_0 \to \Sym^4 X_0$.  (In our case, this took about 8 CPU days to compute and takes over 32 MB to store.)

Next, the closure in the projective plane of the model \eqref{eqn:singmod} has points $(1:1:4)$ and $(1:0:0)$, and we verify that the difference $D \colonequals (1:1:4)-(1:0:0)$ gives $[D] \in A_0(\Q)$ with order $29$: using an algorithmic version of the Riemann--Roch theorem, there exists a rational function $h$ whose divisor is $29D$.  Moreover, $A_0$ has good ordinary reduction at $p=29$---we have $L(X_{\F_{29}},T)=1 + 45T^2 + 2187T^4 + 45\cdot 29^2 T^6 + 29^4 T^8$ and $2187 \equiv 12 \not\equiv 0 \pmod{29}$.  

From now on, to ease notation in the proof we abbreviate $A=(A_0)_F$.  The prime $29$ factors as $29\Z_{K} = \frakp^2 \frakp'$ where $\Nm \frakp = 29$.  Over $F$, since $A$ has RM the natural $29$-adic representation factors as 
\[ \rho_{A,\frakp^\infty} \colon \Galois_F \to \GL_2(K_{\frakp}). \]
Let $\rho_{A,\frakp} \colon \Galois_F \to \GL_2(\F_{\frakp})$ be the associated semisimple representation modulo $\frakp$ as in \eqref{eqn:redmodell}.  We just saw there is a torsion point of order $29$ (in fact already defined over $\Q$), so the semisimplification is reducible: $\rho_{A,\frakp} \simeq \chi_1 \oplus \chi_2$ where $\chi_2$ is the trivial character and $\chi_1 = \varepsilon_{29,F} \colon \Galois_F \to \F_{29}^\times$ is the mod $29$ cyclotomic character over $F$.

We use this to immediately conclude that $A$ has semistable reduction at all primes $\frakq \neq \frakp$.  Indeed, $\rho_{A,\frakp}$ over $F(\zeta_{29})$ is unipotent, so
Grothendieck's inertial semistable reduction criterion applies.

We now verify the hypotheses of Skinner--Wiles \cite[Theorem~A]{SkinnerWiles} in the residually reducible case.  By \Cref{thm:HilbertGalrep}, the representation $\rho_{A,\frakp^{\infty}}$ is unramified away the primes above $29$ and the primes of bad reduction.  This representation is also irreducible, since $\det(1-\rho_{A,\frakp^{\infty}}(\Frob_\frakq)T)$ is irreducible over $K_{\frakp}$ for any prime $\frakq$ of norm $59$.  There is a unique place $v$ of $F$ above $29$.  Moreover,
\[
        F(\chi_1/\chi_2)=F(\zeta_{29}),
\]
which is abelian over $\Q$.  Since this field is totally imaginary, complex conjugation acts nontrivially on it; hence $\chi_1/\chi_2)(c)=-1$ for every complex conjugation $c$.  The character $\chi_1/\chi_2$ is
ramified at $v$, so in particular $(\chi_1/\chi_2)|_{D_v}\ne 1$.  

It remains to check the ordinary local condition at $v$.  The variety $A_0$ has good ordinary reduction at $29$: indeed
\begin{equation}
        L(X_{0,\F_{29}},T)
        =
        1+45T^2+2187T^4+45\cdot 29^2T^6+29^4T^8,
\end{equation} and the middle coefficient satisfies $2187 \equiv 12 \not\equiv 0 \pmod{29}$.  Thus $\rho|_{D_v}$ is crystalline and ordinary.  Let $\widetilde{\chi}_i$ denote the Teichm\"uller lift of $\chi_i$.  After choosing a basis compatible with the ordinary filtration, we have
\begin{equation}
        \rho|_{D_v}
        \simeq
        \begin{pmatrix}
        \psi_1^{(v)}\widetilde{\chi}_1 & * \\
        0 & \psi_2^{(v)}\widetilde{\chi}_2
        \end{pmatrix}.
\end{equation}
Here $\psi_i^{(v)} \equiv 1 \pmod{\frakp}$.  Hence the image of $\psi_2^{(v)}$ lies in $1+\frakp\mathcal \Z_{K_{\frakp}}$, a pro-$29$ group.  Also, the ordinary filtration identifies the upper-left character
with the unit-root quotient, so $\psi_1^{(v)}|_{I_v}$ has finite image.  (For further detail see Brinon--Conrad \cite[\S 8.3, Proposition 8.3.4, Theorem 8.3.6]{BrinonConrad}.)  This is precisely the local hypothesis in Skinner--Wiles, the conclusion of which is that $\rho_{A,\frakp^{\infty}}$ is modular, hence $A=(A_0)_F$ is modular. 

After noting that $\det \rho=\varepsilon_{29^\infty,F}$ is the $29$-adic cyclotomic character restricted to $F$, we conclude that $\rho$ is associated to a Hilbert modular form.  This form must have parallel weight $2$ (by the determinant), and its level is squarefree by semistability and must be supported within the primes of bad reduction of $A_0$.  

This shows that the level of the form divides $\sqrt{3}\frakN$.  At level $\frakN$ the dimension is $22$ and either prime above $11$ distinguishes the newforms and only $f'$ is a match.  At the remaining levels, using eigenvalues at small primes we find no match---the largest space considered is the space of newforms of level $\sqrt{3}\frakN$, of dimension $142$.  
\end{proof}

\begin{remark}
It follows from \Cref{prop:descendtoQQ} and \Cref{prop:modularity} that the Shimura curve $X'$ and the exhibited curve $X_0$ have isogenous Jacobians over $F$, but it unfortunately does not prove that they are isomorphic as curves.  
\end{remark}

\begin{remark}
The methods of Voight--Willis \cite{VW} give another technique for computing equations of Shimura curves (of arbitrary genus).  Since we already had the period lattice, we found numerical reconstruction to be easier here; but we hope to use this technique in future work, as it would for example also provide us (numerically) the images of CM points.  
\end{remark}

In theory, it is now possible to separately compute the $2$-isogeny representation of $\Jac(X)$ and see that it agrees with the one computed numerically from the Hilbert modular form~$f$.  With significant effort, we can also carry this out in practice.  

\begin{prop} \label[prop]{prop:notnicolas}
The field $\Q(A_0[2])$ of $2$-torsion of $A_0$ is the splitting field for \eqref{eqn:fx17}.
\end{prop}

\begin{proof}
We consider the scheme of tritangent planes to $X_0$ in $\PP^3$ (those planes that intersect the degree $6$ curve $X_0$ in three points, each tangent, equivalently effective odd theta characteristics); it is a classical fact that the difference between two tritangent divisors yields all $2$-torsion points of the Jacobian. 
In \textsf{julia}, we compute using homotopy continuation methods \cite{homotopycontinuation} numerical approximations to the nonsingular points of the tritangent scheme \cite{Breiding}.
We then refine these using Newton's method to high precision.
We recover polynomials with real coefficients which vanish on these points and using continued fractions recognize their coefficients as rational numbers.  

With further significant effort factoring these polynomials, we find an \emph{exact} representation of the $120$ tritangent planes over a field of degree $120$.  We then verify using exact methods (\emph{a posteriori}) that these planes are indeed tritangents.  We conclude that the $2$-division field is equal to the splitting field of a degree $120$ polynomial $g(x)$.

Finally, we show that the splitting fields of the polynomials $f(x)$ of degree $17$ in \eqref{eqn:fx17} and $g(x)$ are isomorphic.
Computing the splitting fields separately and checking for isomorphism would be far too expensive.
We start with $K=\Q(\alpha)$ where $f(\alpha)=0$.
We compute using relative invariants for subgroups of permutaton groups (see \cite{Elsenhans,FK})  the degree $3$ extension $K' \supseteq K$ inside the splitting field of $f$; let $K'=\Q(\alpha')$ with $f'(x)$ the minimal polynomial of $\alpha'$.
We are then able to factor $g(x)$ over $K'$ into $3$ irreducible factors of degree $40$; let $h(x) \in K'[x]$ be one of these factors.
Then there exists $H(x,y) \in \Q[x,y]$ such that $H(\alpha',y)=h(y)$  with $\deg_x H(x,y) < 51$ and $\deg_y H(x,y)=40$.  

Now we make a bipartite graph $\Gamma$: 
\begin{itemize}
\item the vertices are the roots $\{\alpha'_i\}_i$ of $f'(x)$ and $\{\beta_j\}_j$ of $g(x)$, respectively, in a splitting field; and
\item there is an (undirected) edge between $\alpha_i'$ and $\beta_j$ if and only if $H(\alpha_i',\beta_j) = 0$.  
\end{itemize}
It is immediate that $\Galois(f'(x)g(x))$ acts on $\Gamma$ via its natural permutation action on the roots, giving an inclusion into $\Aut(\Gamma)$.

The graph $\Gamma$ measures the relationship between the splitting fields that comes about from the factor $h(x)$.
(Two extremes: if $h(x)$ were a linear factor, then the graph would be a simple matching between the roots of $g(x)$ and a subset of roots of $f'(x)$; if at the other extreme $h(x) = g(x)$ and the factor were trivial, then the graph would be a complete bipartite graph giving no new information.)  

The natural projection maps onto permutations of either subset of roots give (surjective) homomorphisms 
\begin{equation}
p_{f'} \colon \Aut(\Gamma) \to \Galois(f'(x)) \quad \text{and} \quad p_g \colon \Aut(\Gamma) \to \Galois(g(x)).
\end{equation}

We compute in this case, working with roots modulo a prime where both polynomials split completely, that in fact each projection $p_{f'}$ and $p_g$ is \emph{injective}!  Therefore we have injective homomorphism $\Galois(f'(x)g(x)) \hookrightarrow \Aut(\Gamma) \xrightarrow{\sim} \Galois(f'(x))$ so that a splitting field of $g$ is contained in that of $f'$; and vice versa, whence they are equal.  
\end{proof}

\begin{remark} \label[remark]{rmk:nicolas}
There is an alternative $p$-adic approach to some of the above computations, as follows.

The division polynomial algorithm of Mascot \cite{Mascot20} takes as input a smooth, projective curve $X$ of genus $g$ with Jacobian $A \colonequals \Jac(X)$, a prime $\ell$, and a prime \hbox{$p \neq \ell$} of good reduction for $X$; it returns as output a rational function $\alpha \in \Q(A)$, a $p$-adic approximation of the corresponding division polynomial $F_\alpha(x) = \prod_{0 \neq P \in A[\ell]}(x - \alpha(P))$, and the matrix $[\Frob_p] \in \GL_{2g}(\F_\ell)$ of the Frobenius automorphism at $p$ acting on $A[\ell]$.  
Running this algorithm in our case with $\ell=2$ and $p=5$ gives a polynomial of degree $2^8-1=255=17 \cdot 15$; with significant computational effort, we verify that it defines a number field that contains the field $K$ defined in \eqref{eqn:fx17}.  

The method of Mascot can be adapted to carve out certain Galois submodules, and these ideas extend to get the degree $17$ polynomial directly, as follows.  We choose a prime $p$ such that the factorization of $F_\alpha(x)$ modulo $p$ consisting of $15$ irreducible factors of degree $17$.  Then in some basis, $\rho(\Frob_p)$ has the form $\begin{pmatrix} \epsilon & 0 \\ 0 & \epsilon^{-1} \end{pmatrix} \in \SL_2(\F_{16}) \leq $ \textrm{17T7}, where $\F_{16}^\times = \langle \epsilon \rangle \simeq C_{15}$.  The smallest such prime is $p=61$.  Since the scalar matrix $\begin{pmatrix} \epsilon & 0 \\ 0 & \epsilon \end{pmatrix}$ centralizes the semisimple element $\rho(\Frob_p)$ in $\GL_2(\F_{16})$, it can be computed explicitly as  $E \in \F_2[\Frob_{p}]$, a polynomial in $\Frob_p$, by linear algebra.  Then from the matrix of $\Frob_p$, we compute the orbits $\Omega_1,\dots,\Omega_{17}$ of $E$ on $A[2] \smallsetminus \{0\}$ and instead form 
\[ \prod_{i=1}^{17} \left(x-\sum_{P \in \Omega_i} \alpha(P)\right) \in \Q_p[x]; \]
good rational approximations yield a polynomial of degree $17$ in $\Q[x]$, which we quickly confirm yields our field $K$.  

This does not give a rigorous result, but in principle it could be made rigorous (working with elements in $A(K)$ using their $p$-adic approximations, and certifying that they are $\ell$-torsion).  

We are grateful to Nicolas Mascot for sharing these calculations and ideas, which given the curve (!)\ takes only about a CPU hour!
\end{remark}


\begin{thebibliography}{BSSVY24}

\bibitem[BSCM23]{BarqueroSanchezCalvoMonge2023}
Adrian Barquero-Sanchez and Jimmy Calvo-Monge,
\emph{On the embedding of Galois groups into wreath products},
2023, preprint, 
\href{https://arxiv.org/abs/2306.14386}{\texttt{arXiv:2306.14386}}.

\bibitem[BDG04]{Bertolini}
Massimo Bertolini, Henri Darmon, and Peter Green,
\emph{Periods and points attached to quadratic algebras},
Heegner points and Rankin $L$-series,
Math.\ Sci.\ Res.\ Inst.\ Publ., \textbf{49},
Cambridge University Press, Cambridge, 2004, 323--367.


\bibitem[BR89]{BlasiusRogawski1989}
D.~Blasius and J.~Rogawski, \emph{Galois representations for {H}ilbert modular forms}, Bull. Amer. Math. Soc. (N.S.) \textbf{21} (1989), vol.~1, 65--69.

\bibitem[vB+24a]{GitHubRepo17T7}
Raymond van Bommel, Edgar Costa, Noam D.\ Elkies, Timo Keller, Sam Schiavone, and John Voight, \emph{17T7},
\url{https://github.com/SamSchiavone/17T7}.

\bibitem[vB+24b]{GitHubRepoEichlerShimuraHMF}
Raymond van Bommel, Edgar Costa, Noam D.\ Elkies, Timo Keller, Sam Schiavone, and John Voight, \emph{EichlerShimuraHMF},
\url{https://github.com/edgarcosta/EichlerShimuraHMF}.

\bibitem[BSSVY24]{BSSVY}
Andrew R.~Booker, Jeroen Sijsling, Andrew V.~Sutherland, John Voight, and Dan Yasaki, \emph{Sato-Tate groups and automorphy for atypical abelian surfaces}, 2024, preprint.

\bibitem[BCP97]{Magma}
Wieb Bosma, John Cannon, and Catherine Playoust, \emph{The Magma algebra system. I. The user language}, J.~Symbolic Comput.\ \textbf{24} (1997), 235--265.

\bibitem[Bos11]{Bosman}
Johan Bosman, \emph{Computations with modular forms and Galois representations}, Computational aspects of modular forms and Galois representations, Ann.\ of Math.\ Stud., vol.~176, Princeton University Press, Princeton, NJ, 2011, 129--157.


\bibitem[Bou23]{Bouchet}
Thomas Bouchet, \textit{Invariants of genus 4 curves}, J. Algebra 660, 2024, 619--644.

\bibitem[Bou24]{bouchet2}
Thomas Bouchet, \emph{Covariant reconstruction of forms from their invariants}, preprint, 2024, \href{https://arxiv.org/abs/2403.17490}{\texttt{arXiv:2403.17490}}.

\bibitem[Bre18]{Breiding}
Paul Breiding, \emph{Tritangent planes to a genus 4 curve}, 26 November 2018, \url{https://www.juliahomotopycontinuation.org/examples/tritangents/} (accessed 21 April 2025).

\bibitem[BC09]{BrinonConrad}
Olivier Brinon and Brian Conrad, \emph{CMI Summer School Notes on $p$-adic Hodge theory}, 2009, {\url{https://claymath.org/sites/default/files/brinon_conrad.pdf}}.

\bibitem[BT18]{homotopycontinuation}
Paul Breiding and Sascha Timme, \emph{HomotopyContinuation.jl: a package for homotopy continuation in Julia}, Mathematical Software – ICMS 2018, eds.\ James H.\ Davenport, Manuel Kauers, George Labahn, and Josef Urban, 2018, Lect. Notes Comp. Sci., vol 10931, Springer, Cham, 458--465.

\bibitem[Bro82]{Brown1982}
Kenneth S.~Brown, \emph{Cohomology of groups}, Grad.\ Texts in Math.,
vol.~87, Springer-Verlag, New York, 1982.

\bibitem[Car86]{Carayol}
Henri Carayol, \emph{Sur les repr\'{e}sentations {$l$}-adiques associ\'{e}es aux formes modulaires de {H}ilbert}, Ann. Sci. \'{E}cole Norm. Sup. (4) \textbf{19} (1986), vol.~3, 409--468.

\bibitem[CDD91]{polredabs}
Henri Cohen and Francisco Diaz y Diaz, \emph{A polynomial reduction algorithm}, 
S\'em.\ Th\'eor.\ Nombres Bordeaux (2) \textbf{3} (1991), no.~2, 351--360.

\bibitem[CMSV19]{endos}
Edgar Costa, Nicolas Mascot, Jeroen Sijsling, and John Voight, \emph{Rigorous computation of the endomorphism ring of a Jacobian}, Math.\ Comp.\ \textbf{88} (2019), 1303--1339. 

\bibitem[Cre97]{Cremona}
J.E.\ Cremona,
{\em Algorithms for Modular Elliptic Curves},
2nd. ed., Cambridge University Press, Cambridge, 1997.

\bibitem[CD17]{CunninghamDembele}
Clifton Cunningham and Lassina Demb\'el\'e, \emph{Lifts of Hilbert modualr forms and application to modularity of abelian varieties}, preprint, 2017, \href{https://arxiv.org/pdf/1705.03054}{\texttt{arXiv:1705.03054}}.

\bibitem[CR06]{CurtisReiner}
Charles~W. Curtis and Irving Reiner, \emph{Representation Theory of Finite Groups and Associative Algebras}, AMS Chelsea Publishing, Providence, RI, 2006.

\bibitem[DL03]{DL}
Henri Darmon and Adam Logan, \emph{Periods of Hilbert modular forms and rational points on elliptic curves}, Int.\ Math.\ Res.\ Not.\ \textbf{2003}, no.~40, 2153--2180. 

\bibitem[Dem08]{Dembele}
Lassina Demb\'el\'e,
\emph{An algorithm for modular elliptic curves over real quadratic fields}, Experiment.\ Math.\ \textbf{17} (2008), no.~4, 427--438.

\bibitem[Dem09]{Dembele-peq2}
Lassina Demb\'el\'e, \emph{A non-solvable Galois extension of ramified at $2$ only}, Comptes Rendus.~Math.\ \textbf{347} (2009), no.\ 3-4, 111--116. 

\bibitem[DeV09]{DV}
Lassina Demb\'el\'e and John Voight, \emph{Explicit methods for Hilbert modular forms}, Elliptic curves, Hilbert modular forms and Galois deformations, Birkh\"{a}user, Basel, 2013, 135--198.

\bibitem[DGV11]{DGV}
Lassina Demb\'el\'e, Matthew Greenberg, and John Voight, \emph{Nonsolvable number fields ramified only at 3 and 5}, Compositio Math.\ \textbf{147} (2011), no.~3, 716--734.

\bibitem[Del71]{Deligne}
Pierre Deligne, \emph{Travaux de Shimura}, S\'eminaire Bourbaki, 23\`eme ann\'ee (1970/1971), expos\'e 389, Lecture Notes in Math., vol.~244, Springer, 1971, 123--165.

\bibitem[DN67]{DN}
Koji Doi and Hidehisa Naganuma, \emph{On the algebraic curves uniformized by arithmetical automorphic functions}, Ann.~Math.~(2) \textbf{86} (1967), no.~3, 449--460.

\bibitem[Dok04]{Dokchitser-lfun}
Tim Dokchitser, \emph{Computing special values of motivic $L$-functions}, Experiment.\ Math.\ \textbf{13} (2004), no.~2, 137--149.

\bibitem[Dok21]{Dokchitser}
Tim Dokchitser, \emph{Inverse Galois problem: PCMI 2021 Graduate Summer School Program ``Number Theory Informed by Computation''}, 2021, \url{https://people.maths.bris.ac.uk/~matyd/InvGal/}.


\bibitem[DoV21]{DonV}
Steve Donnelly and John Voight, \emph{A database of Hilbert modular forms}, Arithmetic geometry, number Theory, and computation, eds.\ Jennifer S. Balakrishnan, Noam Elkies, Brendan Hassett, Bjorn Poonen, Andrew V.\ Sutherland, and John Voight, Simons Symp., Springer, Cham, 2021, 365--373.

\bibitem[Edi11]{Edixhoven}
Bas Edixhoven, \emph{Computing the residual Galois representations}, Computational aspects of modular forms and Galois representations, Ann. of Math. Stud., vol.~176, Princeton Univ. Press, Princeton, NJ, 2011, 371--382.

\bibitem[EK]{EK}
Noam D. Elkies and Jean Kieffer, \emph{A uniform quasi-linear time algorithm for evaluating theta functions in any dimension}, preprint, 
\href{https://arxiv.org/abs/2505.22382}{\texttt{arXiv:2505.22382}}.

\bibitem[Els17]{Elsenhans}
Andreas-Stephan Elsenhans, \emph{Improved methods for the construction of relative invariants for permutation groups}, J. Symbolic Comput. {\bf 79} (2017), 211--231.

\bibitem[FK14]{FK}
Claus Fieker and J\"urgen Kl\"uners, \emph{Computation of Galois groups of rational polynomials}, LMS J. Comput. Math. \textbf{17} (2014), no.~1, 141--158.

\bibitem[Flint]{Flint}
The FLINT team, \emph{FLINT: Fast Library for Number Theory}, 2024, Version 3.1.3, \url{https://flintlib.org}.

\bibitem[G88]{VanderGeer}
Gerard van der Geer,
\emph{Hilbert Modular Surfaces},
Ergeb.\ Math.\ Grenzgeb.\ (3), vol.~16,
Springer-Verlag, Berlin, 1988.

\bibitem[GP17]{GP17}
Jose Ignacio Burgos Gil and Ariel Pacetti, \emph{Hecke and Sturm bounds for Hilbert modular forms over real quadratic fields}, Math.\ Comp.\ {\bf 86} (2017), no.~306, 1949--1978.

\bibitem[Gor02]{Goren}
Eyal Z.\ Goren, \emph{Lectures on Hilbert Modular Varieties and Modular Forms},
CRM Monogr.\ Ser., vol.~14, Amer. Math. Soc., Providence, RI, 2002.

\bibitem[Gra96]{Granboulan}
Louis Granboulan, \emph{Construction d'une extension r\'eguli\`ere de ${\bf Q}(T)$ de groupe de Galois $M_{24}$}, Experiment. Math. \textbf{5} (1996), no.~1, 3--14.

\bibitem[GV11]{GV}
Matthew Greenberg and John Voight, \emph{Computing systems of Hecke eigenvalues associated to Hilbert modular forms}, Math. Comp. {\bf 80} (2011), no.~274, 1071--1092.

\bibitem[HPS24]{HPS}
Jeroen Hanselman, Andreas Pieper, and Sam Schiavone, 
\emph{Equations of genus 4 curves from their theta constants}, preprint, 2024, \href{https://arxiv.org/abs/2402.03160}{\texttt{arXiv:2402.03160}}.

\bibitem[Igu64]{Igusa64}
Jun-ichi Igusa,
\emph{On Siegel modular forms of genus two. II},
Amer. J. Math. {\bf 86} (1964), 392--412.

\bibitem[JLY02]{JLY}
Christian U.~Jensen, Arne Ledet, and Noriko Yui, \emph{Generic polynomials:  Constructive aspects of the inverse Galois problem},  Math. Sci. Res. Inst. Publ.,  vol.~45, Cambridge University Press, Cambridge, 2002.

\bibitem[JR07]{JonesRoberts}
John W. Jones and David P. Roberts. \textit{Galois number fields with small root discriminant}, Journal of Number Theory \textbf{122} (2007), no.~2, 379-407.

\bibitem[Kin05]{King}
Oliver H.~King, \emph{The subgroup structure of finite classical groups in terms of geometric configurations}, Surveys in combinatorics 2005, London Math. Soc. Lecture Note Ser., vol.~327, Cambridge Univ. Press, Cambridge, 2005, 29--56.

\bibitem[KM01]{KlunersMalle}
J\"{u}rgen Kl\"{u}ners and Gunter Malle, \emph{A database for field extensions of the rationals}, LMS J. Comput. Math. \textbf{4} (2001), 182--196.

\bibitem[KM24]{KMweb}
J\"urgen Kl\"uners and Gunter Malle, \emph{Missing polynomials (group)}, 2024, \url{http://galoisdb.math.upb.de/statistics/missing_polynomials/group}.

\bibitem[LMFDB]{LMFDB}
The LMFDB Collaboration, \emph{The L-functions and modular forms database}, \url{https://www.lmfdb.org}, accessed 2 April 2024.


\bibitem[MM99]{MM}
Gunter Malle and B.\ Heinrich Matzat, \emph{Inverse Galois Theory}, Springer Monogr. Math., Springer-Verlag, Berlin, 1999.

\bibitem[Mas18]{Mascot}
Nicolas Mascot, \emph{Certification of modular Galois representations}, 
Math.\ Comp.\ \textbf{87} (2018), no.~309, 381--423.

\bibitem[Mas20]{Mascot20}
Nicolas Mascot, \emph{Hensel-lifting torsion points on Jacobians and Galois representations}, Math.\ Comp.\ \textbf{89} (2020), no.~323, 1417--1455.

\bibitem[Oda82]{Oda}
Takayuki Oda, \emph{Periods of Hilbert modular surfaces}, Progr. Math., vol.~19
Birkh\"auser, Boston, MA, 1982.

\bibitem[Oda90]{Oda2}
Takayuki Oda,
\textit{The Riemann-Hodge period relation for Hilbert modular forms of weight 2}. Cohomology of arithmetic groups and automorphic forms (Luminy-Marseille, 1989), Lecture Notes in Math., vol.~1447, Springer-Verlag, Berlin, 1990, 261--286.

\bibitem[Pari]{pari}
    The PARI~Group, PARI/GP version \texttt{2.15.4}, Univ. Bordeaux, 2023,
    \url{http://pari.math.u-bordeaux.fr/}.
    
\bibitem[Ric22]{Rickards}
James Rickards, \emph{Improved computation of fundamental domains for arithmetic Fuchsian groups}, Math.\ Comp.\ \textbf{91} (2022), no.~338, 2929--2954.

\bibitem[Rot95]{Rotman1995}
Joseph J.~Rotman, \emph{An introduction to the theory of groups},
4th~ed., Grad.\ Texts in Math., vol.~148, Springer-Verlag, New York, 1995.

\bibitem[Ser08]{Serre}
Jean-Pierre Serre, \emph{Topics in Galois theory}, 2nd.~ed., with notes by Henri Darmon, Res. Notes Math., vol.~1, A K Peters, Ltd., Wellesley, MA, 2008. 

\bibitem[Shi67]{Shimura}
Goro Shimura, \emph{Construction of class fields and zeta functions of algebraic curves}, Ann.~Math.~(2), \textbf{85} (1967), no.~1, 58--159.

\bibitem[SV16]{SV}
Jeroen Sijsling and John Voight, \emph{On explicit descent of marked curves and maps}, Res.\ Number Theory \textbf{2}:27 (2016), 35 pages. 

\bibitem[SW98]{SkinnerWiles}
C.M.~Skinner and Andrew J.~Wiles, \emph{Residually reducible representations and modular forms}, Inst.\ Hautes\ \'Etudes Sci.\ Publ.\ Math., no. 89, 1999, 5--126.

\bibitem[Sta73]{Stauduhar}
Richard P.~Stauduhar, \emph{The determination of Galois groups}, Math.\ Comput.\ \textbf{27} (1973), 981--996.

\bibitem[Tay89]{Taylor}
Richard Taylor, \emph{On {G}alois representations associated to {H}ilbert modular forms}, Invent.\ Math.\ \textbf{98} (1989), vol.~2, 265--280.

\bibitem[Voi09]{Voi:funddom}
John Voight, \emph{Computing fundamental domains for Fuchsian groups}, J.~Th\'eorie Nombres Bordeaux \textbf{21} (2009), no.~2, 467--489.

\bibitem[Voi21]{Voight:quat}
John Voight, \emph{Quaternion algebras}, Grad.\ Texts in Math., vol.~288, Springer, Cham, 2021. 

\bibitem[VW14]{VW}
John Voight and John Willis, \emph{Computing power series expansions of modular forms}, Computations with modular forms, Contrib. Math. Comput. Sci., vol.~6, Springer, Cham, 2014, 331--361.

\bibitem[Z01]{Zhang}
Shouwu Zhang, \emph{Heights of Heegner points on Shimura curves}, Ann. of Math. (2) {\bf 153} (2001), no.~1, 27--147.

\bibitem[Zyw23]{Zywina}
David Zywina, \emph{Modular forms and some cases of the inverse Galois problem}, Canad.\ Math.\ Bull.\ \textbf{66} (2023), no.~2, 568--586.

\end{thebibliography}
\end{document}